\documentclass[11pt,reqno]{amsart}
\usepackage{amssymb}
\usepackage[curve,matrix,arrow]{xy}

 \newtheorem{thm}{Theorem}[section]
\newtheorem{prop}[thm]{Proposition}

\newtheorem{lemma}[thm]{Lemma}
\newtheorem{Th}{Theorem}

\numberwithin{equation}{section}

\theoremstyle{definition}
\newtheorem{rem}[thm]{Remark}
\newtheorem{defn}[thm]{Definition}

\newcommand{\ls}[2]{~{}^{#1}\!{#2}}
\newcommand{\wh}[1]{\widehat{#1}}

\def\F{\mathbb{F}}
\def\bF{\mathbb{F}}
\def\G{\mathbb G}
\def\H{\mathbb H}
\def\K{\mathbb K}

\def\Z{\mathbb Z}

\def\CA{{\mathcal A}}

\def\Det{\operatorname{Det}\nolimits}

\def\Hom{\operatorname{Hom}\nolimits}

\def\SL{\operatorname{SL}\nolimits}
\def\GL{\operatorname{GL}\nolimits}
\def\PSL{\operatorname{PSL}\nolimits}
\def\PGL{\operatorname{PGL}\nolimits}
\def\Sp{\operatorname{Sp}\nolimits}
\def\PSp{\operatorname{PSp}\nolimits}

\def\SU{\operatorname{SU}\nolimits}
\def\PGU{\operatorname{PGU}\nolimits}
\def\GU{\operatorname{GU}\nolimits}
\def\PSU{\operatorname{PSU}\nolimits}
\def\CSp{\operatorname{CSp}\nolimits}

\def\SO{\operatorname{SO}\nolimits}

\def\Spin{\operatorname{Spin}\nolimits}

\usepackage{mathrsfs}

\def\D{\mathbb{D}}
\newcommand{\tuborg}{\left\{\begin{array}{ll}}
\newcommand{\sluttuborg}{\end{array}\right.}
\def\rk{\operatorname{rk}\nolimits}

\def\S{\mathbb S}
\usepackage{cleveref}

\newcommand{\lcom}{\hat{{}_\ell}}

\newcommand{\xrightcong}{\xrightarrow{\raisebox{-0.4ex}[0ex][0ex]{$\,
      \Small \cong\, $}}}

\newcommand{\TF}{TF}

\usepackage{etoolbox,refcount}
\usepackage{multicol}

\newcounter{countitems}
\newcounter{nextitemizecount}
\newcommand{\setupcountitems}{%
  \stepcounter{nextitemizecount}%
  \setcounter{countitems}{0}%
  \preto\item{\stepcounter{countitems}}%
}
\makeatletter
\newcommand{\computecountitems}{%
  \edef\@currentlabel{\number\c@countitems}%
  \label{countitems@\number\numexpr\value{nextitemizecount}-1\relax}%
}
\newcommand{\nextitemizecount}{%
  \getrefnumber{countitems@\number\c@nextitemizecount}%
}
\newcommand{\previtemizecount}{%
  \getrefnumber{countitems@\number\numexpr\value{nextitemizecount}-1\relax}%
}
\makeatother    
\newenvironment{AutoMultiColItemize}{%
\ifnumcomp{\nextitemizecount}{>}{3}{\begin{multicols}{2}}{}%
\setupcountitems\begin{itemize}}%
{\end{itemize}%
\unskip\computecountitems\ifnumcomp{\previtemizecount}{>}{3}{\end{multicols}}{}}

\title[Torsion Free Endotrivial modules]{Torsion Free Endotrivial modules for finite groups of Lie type}

\author{\sc Jon F. Carlson}
\address
{Department of Mathematics\\ University of Georgia \\
Athens\\ GA~30602, USA}
\email{jfc@math.uga.edu}
\author{\sc Jesper Grodal}
\address
{Department of Mathematical Sciences\\ University of
  Copenhagen\\ Copenhagen \\DK-2100, Denmark} 
\email{jg@math.ku.dk}
\author{\sc Nadia Mazza}
\address
{Department of Mathematics and Statistics\\ Lancaster University\\ Lancaster \\LA1 4YF, UK} 
\email{n.mazza@lancaster.ac.uk}
\author{\sc Daniel K. Nakano}
\address
{Department of Mathematics\\ University of Georgia \\
Athens\\ GA~30602, USA}
\email{nakano@math.uga.edu}


\thanks{The first author was partially supported by Simons Foundation 
  under Grant No.\ 054813-01. The second author was supported by the
  Danish National Research Foundation through the Centre for Symmetry
  and Deformation (DNRF92) and the Copenhagen Center for Geometry and
  Topology (DNRF151). The fourth author was partially supported by 
NSF grant number DMS-1701768.}

\begin{document}
\begin{abstract} In this paper we determine the torsion free
rank of the group of endotrivial modules for any
finite group of Lie type, in both defining and non-defining
characteristic. Equivalently, we classify the maximal
rank $2$ elementary abelian $\ell$-subgroups in any finite group of
Lie type, for any prime $\ell$. This classification  
may be of independent interest.
\end{abstract}

\maketitle


\section{Introduction} \label{sec:intro} 
Endotrivial modules play a significant role in 
the modular representation theory of finite groups; in particular,
they are the invertible elements in the Green ring of the stable
module category of finitely generated modules for the group algebra.  
Tensoring with an endotrivial module is a self equivalence of the
stable module category and these operations
generate the Picard group of self equivalences of Morita type in this
category.
The endopermutation modules, defined for finite groups of
prime power order, are the sources of the irreducible modules for
large classes of finite groups, and these endopermutation modules are
built from the endotrivial modules.  

Let $G$ be a finite group and let $k$ be a field of prime
characteristic $\ell$ that divides the order of $G$.  
A finitely generated $kG$-module $M$ is {\em endotrivial}
if its $k$-endomorphism ring $\Hom_k(M,M)$ is the direct sum
of a trivial module and a projective module. The isomorphism classes 
in the stable category of such modules 
form an abelian group $T(G)$
under the tensor product $\otimes_k$, where $M\otimes_kN$ is
equipped with the diagonal $G$-action. The group has identity $[k]$ and 
the inverse to a class $[M]$ is the class $[M^*]$, where $M^*$
is the $k$-dual of $M$.
As $T(G)$ is finitely generated it is isomorphic to the
direct sum of its torsion subgroup $TT(G)$, and
a finitely generated torsion free group $\TF(G) = T(G)/TT(G)$. We define the
{\em torsion free rank} of $T(G)$ to be the rank of $\TF(G)$ 
as a $\mathbb Z$-module. 
In \cite{Gro}, the second author used homotopy theory to describe
$TT(G)$, tying the structure of $TT(G)$ to that of $G$ itself, and in
doing so, he also proved a conjecture by the first author and
Th\'evenaz \cite{CT}. 
In a forthcoming article \cite{CGMN2}, we will provide a 
description of the torsion subgroup $TT(G)$ for $G$
a finite group of Lie type for all primes,
using homotopy theoretic methods. For more information on the history and 
applications of endotrivial modules, see the survey papers 
\cite{CRL, Th}, and the book by the third author \cite{Mazza}. 

We recall that, for any finite group $G$, there is a distinguished element in
$T(G)$, namely the class of the shift of the trivial module, defined to be the kernel of
the map from a projective cover of $k$ to $k$. It is easily verified to be
endotrivial. Moreover, by elementary
homological algebra, the class of this  element has infinite order in $\TF(G)$
if and only if $G$ contains a subgroup 
isomorphic to $\Z/\ell \times \Z/\ell$.

Our main theorem of this paper determines the rank of $\TF(G)$
for $G$ any finite group of Lie type of characteristic
$p$. We show that it is generated by the class of the shift of the trivial
module except in a few low-rank cases, that we describe
explicitly. Before stating the precise version of the main 
theorem, we need to make clear 
what we mean by a finite group of Lie type.

\begin{defn}[Finite group of Lie type]\label{def:lietype}
 By a {\em finite group of Lie type} in
characteristic $p$ we mean a group $G = \G^F$ for $\G$ a connected
reductive algebraic group over an algebraically closed field of
characteristic $p$, and $F$ a Steinberg endomorphism, i.e., an
endomorphism of $\G$ such that $F^s$ is a standard Frobenius map
$F_q$, for $q = p^r$ and some $s,r \geq 1$.
\end{defn}

This definition is a bit more general than 
that of \cite[Definition 21.6]{MT} in that we
only assume $\G$ to be reductive instead of semisimple. For example, 
this includes the classical group $\GL_n(q)$. 
We now present our main theorem:

\begin{Th} \label{thm:tf-reductive-main}
Let $G$ be a finite group of Lie type in characteristic $p$ as in Definition~\ref{def:lietype}.
The group $\TF(G)$ of torsion free endotrivial modules
over a field of characteristic $\ell$, with $\ell \mid |G|$, 
is zero or infinite cyclic generated by the class of the shift of the trivial
module, except when $G$ is on the following list:
\begin{enumerate}
\item \label{case-cross} $\ell \neq p$ and
$G \cong H \times K$, where $\ell \nmid |K|$, and $H$ is 
either

\begin{enumerate}
\item \label{PGL}$\PGL_{\ell}(q)$ with $\ell  \mid q-1$,
  \item \label{PGU} $\PGU_{\ell}(q)$ with
    $\ell \mid q+1$, or
    \item\label{3D4}$\ls3D_{4}(q)$ with $\ell = 3$.
\end{enumerate}
\item \label{case-char} $\ell =p$ 
and $G/Z(G)$ is either $\PSU_3(p)$ for $p
\geq 3$ and  $3 \mid p+1$, $\PSL_3(p)$ for $p \geq 2$,
$\PGL_3(p)$ for $p \geq 2$, $\operatorname{PSpin}_5(p)$ 
for $p \geq 5$,   $\SO_5(p)$
for $p \geq 5$, or  $G_2(p)$ for $p \geq  7$.
\end{enumerate}
In case \eqref{case-cross}, $\TF(G) \xrightcong \TF(H)$ has rank $3$ if
$H \cong \PGL_{\ell}(q)$ or  $\PGU_{\ell}(q)$ and $\ell >
2$, and rank $2$ if $\ell =2$ or $H \cong \ls3D_{4}(q)$;
see Theorems~\ref{thm:lgeq3} and \ref{thm:l=2}. In case
\eqref{case-char} the ranks are listed in the tables in 
Section~\ref{sec:defining}; see Theorem \ref{thm:l=p}.  
\end{Th}

The quotient groups $G/Z(G)$ occurring above as the classical groups
$\PSL_3(p) = \SL_3(p)/C_3$,
$\PSU_3(p) = \SU_3(p)/C_3$,
and $\operatorname{PSpin}_5(p) = \Spin_5(p)/C_2$
are in fact not themselves finite groups of
Lie type; see Remark~\ref{rem:lietype} and Section~\ref{sec:assoc}
for more about this subtlety.  Section~\ref{sec:assoc} also contains
analogous results for all groups of the form $\G^F/Z(\G^F)$, for simply
connected simple $\G$, i.e., the {\em finite simple groups} associated
to finite groups of Lie type.
Special cases of the above results can be found in
\cite{CMN,CMN3,CMN4}.
Note that the rank of  $\TF(G)$ depends on the
characteristic $\ell$ of $k$, but not  on the finer
structure of $k$.

An elementary abelian $\ell$-subgroup of $G$ is a subgroup
isomorphic to an $\mathbb F_\ell$-vector space. Its  $\ell$-rank is
its $\mathbb F_\ell$-vector space dimension. The $\ell$-rank of $G$, denoted
$\rk_\ell(G)$, is the maximum of the $\ell$-ranks of elementary abelian
$\ell$-subgroups of $G$. The groups in \eqref{PGL} and \eqref{PGU} of  Theorem~\ref{thm:tf-reductive-main} have
$\ell$-rank $\ell-1$ when $\ell$ is odd, while all other 
groups listed in  \eqref{case-cross} and \eqref{case-char} have $\ell$-rank $2$. 

By a well-known correspondence, recalled 
in Theorem~\ref{thm:poset2} below, our main
result translates into a purely local group theoretic
statement, Theorem~\ref{thm:main-grp}, which is in fact what we prove.
Let $\CA^{\geq 2}_\ell(G)$ denote the poset of noncyclic elementary
abelian $\ell$-subgroups of $G$, ordered by
subgroup inclusion. We say that an elementary abelian $\ell$-subgroup
of $G$ is maximal if it is maximal in $\CA^{\geq 2}_\ell(G)$, i.e., if it
is not properly contained in any other elementary abelian subgroup of
$G$.
The poset $\CA^{\geq 2}_\ell(G)$ has a $G$-action by conjugation, and we can also
consider the orbit space $\CA^{\geq 2}_\ell(G)/G$. For
any poset $X$, we can define its set of connected components
$\pi_0(X)$, as equivalence classes of elements generated by the order
relation, and note that, for a $G$-poset, 
$\pi_0(X)/G \xrightcong \pi_0(X/G)$.
The following theorem states the correspondence. 

\begin{thm}[{\cite[Theorem 4]{A2}  
\cite[Theorem 3.1]{CMN}}]\label{thm:poset2}
For any finite group $G$ and prime $\ell$ dividing the order of $G$, the rank of the group  
$\TF(G)$ is equal 
to the number of connected components of
the orbit space $\CA^{\geq 2}_\ell(G)/G$. 
This number is $0$ if $\rk_\ell(G) =1$; it is equal to 
the number of conjugacy classes of maximal elementary abelian
$\ell$-subgroups in $G$ if $\rk_\ell(G) =2$; and 
it is equal to $1$ more than the number of
conjugacy classes of maximal elementary abelian
$\ell$-subgroups of rank $2$, if  $\rk_\ell(G) >2$.
\end{thm}

The theorem above is Alperin's \cite{A2} original calculation of
the torsion free rank of $T(G)$ in the case that $G$ is a finite
$\ell$-group. The proof for arbitrary finite groups 
is given in \cite{CMN} and uses very different methods. 
With this dictionary in place, we can state a local 
group theoretic version of our main result:

\begin{Th}\label{thm:main-grp}
Let $G$ be a finite group of Lie type in characteristic $p$ 
(see Definition~\ref{def:lietype}) and $\ell$ an arbitrary prime.
\begin{enumerate}
\item
If $\rk_\ell(G)>2$, then $G$ does
not have a maximal elementary abelian $\ell$-subgroup of rank $2$, unless
$\ell>3$, $\ell \neq p$, and $G$ has the
form given in Theorem~\ref{thm:tf-reductive-main}\eqref{PGL} or
\eqref{PGU} (where  $\rk_\ell(G) = \ell-1$). 
\item
If $\rk_\ell(G) =2$, then all elementary abelian $\ell$-subgroups of
$G$ of rank $2$ are conjugate unless 
$G$ has the form given in
Theorem~\ref{thm:tf-reductive-main}\eqref{case-char},  in
Theorem~\ref{thm:tf-reductive-main}\eqref{3D4}, or in
 Theorem~\ref{thm:tf-reductive-main}\eqref{PGL}\eqref{PGU}, $\ell \leq
 3$.
\end{enumerate}
\end{Th}

To provide additional context to Theorem~\ref{thm:main-grp}, 
recall that $G$ can only have a maximal
elementary abelian $\ell$-subgroup of rank $2$ when 
$\rk_\ell(G) \leq \ell$ for $\ell$ odd, and
$\rk_2(G) \leq 4$ when $\ell =2$, by a theorem of 
Glauberman--Mazza \cite{GlMa} and MacWilliams \cite{McW}
(restated as Theorem~\ref{T:norank2}). Theorem~\ref{thm:main-grp} pins
down exactly the cases where this does in fact occur for finite
groups of Lie type.
The study of elementary abelian $\ell$-subgroups of $\G$ and $\G^F$
has a long history, with close relationship to 
cohomology and representation theory; see e.g.,
\cite{borel61,  BFM02, quillen71, quillen78, steinberg75}. When $\ell
\neq p$, conjugacy classes of elementary abelian 
$\ell$-subgroups of $\G$ identify with those
of the corresponding complex reductive algebraic group, or compact Lie
group (see \cite[{Section~8}]{AGMV08}).
In fact, they only depend on the $\ell$-local structure 
as encoded in the $\ell$-compact group
$(B\G)\lcom$ obtained by $\ell$-completing the classifying space
$B\G$ in the sense of homotopy theory \cite{grodal10}. 
Similarly, the elementary abelian $\ell$-subgroups of $G$ are
determined by $BG\lcom$, an $\ell$-local finite group \cite{BLO03} describable
from the action of $F$ on $B\G\lcom$; see e.g.,
\cite[Appendix C]{GL20} for a summary.
The question of existence of maximal rank $2$ elementary abelian
$\ell$-subgroups can thus be asked more generally in the context of
homotopy finite groups of Lie type,  i.e., 
homotopy fixed-points of Steinberg
endomorphisms on connected $\ell$-compact groups
\cite{BM07, GL20}. In fact we expect Theorem~\ref{thm:main-grp} to
generalize to this setting, with the same conclusion, as
simple $\ell$-compact groups not coming from a compact connected Lie
group are centerless and have a unique maximal elementary abelian
$\ell$-subgroup (see \cite[Theorems~1.2 and 1.8]{AGMV08}
and \cite[Theorem~1.1]{AG09}).  We do not pursue the details
here, but see Remark~\ref{rem:lcgsubgroups}.

One may similarly wonder if $\TF(G)$ of 
Theorem~\ref{thm:tf-reductive-main} only depends on the
$\ell$-local structure in the stronger sense that if $H \to G$ induces
an isomorphism of $\ell$-fusion systems, is the  {\em map} $\TF(G) \to \TF(H)$
an isomorphism? That question, however,
has a negative answer in
general, and we need to replace $\ell$-fusion by a stronger
$\ell$-local invariant \cite{BGH20}. 

\subsubsection*{Structure of the paper} 
Section~\ref{sec:endotriv} collects background results needed later,
including the aforementioned general Theorem~\ref{T:norank2} that gives conditions
on $\rk_\ell(G)$ ensuring
no maximal elementary abelian $\ell$-subgroups of rank $2$.

In Sections~\ref{sec:lg3generic}--\ref{sec:defining}, we 
determine $\TF(G)$ when $G = \G^F$, and $\G$ is simple. 
The cases when $3 \leq \ell \neq p$ are 
handled in Sections~\ref{sec:lg3generic} and \ref{sec:lg3specific}.  
In many cases it is known that the orbit space $\CA^{\geq 2}_\ell(G)/G$
is connected (see \cite[Section~4.10]{GLS}). 
This allows us to reduce to examining some
groups of small Lie rank, 
in Proposition~\ref{prop:oddreduce}, and these are then 
analyzed in Section~\ref{sec:lg3specific}. In Section~\ref{sec:assoc}, we extend the results of the previous
sections to also  compute $\TF(G)$, for $G$ a group closely associated to 
a group of Lie type such as $\PSL_n(q)$ or $\PSp_n(q)$, in the case
that $\ell \geq 3$.

The case where  $2 = \ell \neq p$ is handled in
Section~\ref{sec:le2generic}.
Section~\ref{sec:defining} investigates the final case when
${\ell}=p$, extending work in \cite{CMN}.
 In the case that $\ell = 2$ the associated groups
are included in the analysis of Section~\ref{sec:le2generic}. 

Finally, in Section~\ref{sec:reductive}, 
we prove Theorems~\ref{thm:tf-reductive-main} and ~\ref{thm:main-grp}
in the general case where $\G$ is a connected reductive algebraic group. 

\subsubsection*{Acknowledgments} In the course of writing 
this paper, the authors checked many examples 
using the computer algebra system Magma \cite{magma}. 
We thank John Cannon 
and his team for providing this wonderful tool. 
The authors also acknowledge Gunter Malle for helpful 
conversations at various stages throughout this project. 
In particular we thank him for his detailed comments on an earlier
version of this manuscript that, among other things, clarified the
treatment of the very twisted groups in Section~\ref{sec:lg3generic}.


\section{Preliminaries}\label{sec:endotriv}
\begin{verse}
{\em Throughout the paper $G$ is
finite group (maybe subject to more assumptions) and $k$ is a field of some positive characteristic
$\ell$, dividing the order of $G$}.
\end{verse}

In the rest of section we provide some background material used throughout this paper.
\begin{defn}
A finitely generated $kG$-module $M$ is {\em endotrivial} if
$\Hom_k(M, M)\cong k \oplus P$ where $P$ is a projective
$kG$-module and $k$ is the trivial $kG$-module.
Thus, $\Hom_k(M,M) \cong k$ in the stable category of 
$kG$-modules modulo projectives.
The set $T(G)$ of stable isomorphism classes of endotrivial $kG$-modules
forms a group under $-\otimes_k-$, called the
{\em group of endotrivial $kG$-modules}.
\end{defn}

Recall that in this context, $\Hom_k(M,M)\cong M^*\otimes_kM$ as
$kG$-modules, and therefore the endotrivial modules are the invertible
objects under tensor product in the stable module category of $kG$-modules
modulo projectives.

The group $T(G)$ is a finitely generated 
abelian group (\cite[Corollary 2.5]{CMN}) hence
$T(G) \cong TT(G)\oplus \TF(G)$, for $TT(G)$ the torsion subgroup of $T(G)$, a
finite group, and $\TF(G) = T(G)/TT(G)$, a
finitely generated free abelian group.
In Theorem~\ref{thm:poset2}, the rank of $\TF(G)$
is stated to be equal to the number of conjugacy classes of maximal
elementary abelian $\ell$-subgroups of $G$ of rank $2$ if
$\rk_\ell(G)=2$, or that number plus one in case $\rk_\ell(G)>2$.

We start with a few elementary but useful observations.

\begin{lemma}  \label{lem:normal3rk} \label{lem:ele1} Let $P$ be a
  finite $\ell$-group. 
\begin{itemize}
\item[(a)] If $P$ has a normal elementary abelian $\ell$-subgroup $H$
of $\ell$-rank $\ell+1$ or more, then $P$ has no maximal elementary
abelian subgroups of rank $2$.
\item[(b)] If $P$ has $\ell$-rank $2$ and the center of $P$ is not cyclic,
then $P$ has exactly one maximal elementary
abelian subgroup with $\ell$-rank $2$.
\item[(c)] If $P$ has $\ell$-rank at least $3$ and the center of $P$ is not
cyclic, then $P$ has no maximal elementary abelian subgroups of $\ell$-rank
$2$. 
\end{itemize}
\end{lemma}

\begin{proof}
For (a), let $x\in P$ be an element of order $\ell$ in a 
maximal elementary abelian subgroup. If $x\in H$,
then $C_P(x)\geq H$ has $\ell$-rank at least $3$ by assumption and the
statement holds.
If $x\notin H$, then
the conjugation action of $x$ on $H$ can be regarded as a linear action on an
$\bF_\ell$-vector space of dimension at least $\ell+1$, and therefore must
have at least two linearly independent eigenvectors for the eigenvalue
$1$. That is, conjugation by $x$ fixes two nontrivial distinct
generators of $H$ in some suitable generating set,
and since $x\notin H$, we conclude that the subgroup of $P$ generated
by $x$ and these two elements is elementary abelian of rank $3$. So
$x$ is not contained in a maximal elementary abelian subgroup of $P$
of rank $2$, and part (a) follows.
The proofs of parts (b) and (c) are straightforward. 
\end{proof}

For our analysis, we employ results of Glauberman--Mazza and 
MacWilliams that guarantee, under suitable conditions 
on the $\ell$-rank of the finite group $G$, that the group has no 
maximal elementary abelian $\ell$-subgroups of rank $2$. 
The sectional $\ell$-rank of a group $G$ is the maximal $\ell$-rank of
any section of $G$. A section of $G$ is the quotient of a subgroup of $G$
by a normal subgroup of that subgroup.

\begin{thm} \label{T:norank2} 
Let $G$ be a finite group and let ${\ell}$ be a prime. 
\begin{itemize}
\item[(a)] \cite[Theorem A]{GlMa} If $\ell\geq 3$ and 
$\rk_{\ell}(G)\geq {\ell}+1$, then 
$G$ has no maximal elementary abelian $\ell$-subgroups of rank $2$. 
\item[(b)] \cite[Four Generator Theorem]{McW} 
Suppose that $G$ has sectional $2$-rank at least $5$.
Then a Sylow $2$-subgroup of $G$ has a normal elementary abelian
subgroup with $2$-rank $3$. In such a case $G$ has no maximal
elementary abelian $2$-subgroup of rank $2$.
\end{itemize} 
\end{thm} 

Part (b) in Theorem~\ref{T:norank2} is a reformulation,
which better suits our analysis, of \cite[Four
Generator Theorem]{McW}. The theorem (which was part of the program 
to classify finite simple groups) asserts that, in a finite
$2$-group $G$ with no normal elementary abelian subgroup of rank $3$,
every subgroup can be generated by at most four elements. 
Thus, if the sectional $2$-rank of a $2$-group $G$ is $5$ or more, then some  
Frattini quotient $P/\Phi(P)$, for $P$ a subgroup
of $G$, has $2$-rank $5$ or more. By the theorem, $G$ has a normal         
elementary abelian subgroup with $2$-rank $3$, implying that
$G$ has no maximal elementary abelian subgroup of rank $2$, 
by Lemma  \ref{lem:normal3rk}. Our
interpretation follows because, for any $\ell$, the sectional $\ell$-rank of
a finite group is equal to that of its Sylow $\ell$-subgroups.

We also record the following result, which is used to relate 
the torsion free ranks of groups of endotrivial modules of finite groups
of Lie type arising from isogenous algebraic groups.

\begin{prop} \label{P:groupexactseq}
Let 
$$
\xymatrix{1\ar[r]&Z\ar[r]&H\ar[r]&G\ar[r]&K\ar[r]&1}
$$
be an exact sequence of finite groups where $Z$ and $K$ have order prime to
$\ell$, and $Z$ central in $H$. Then the induced map
$\CA^{\geq
    2}_\ell(H)/H \twoheadrightarrow \CA^{\geq 2}_\ell(G)/G$ is a
  surjection, which is an isomorphism of posets if the image of $H$ in $G$ controls $\ell$-fusion in $G$.
 In particular $\TF(H) \cong \Z$ implies $\TF(G) \cong \Z$, with the
 converse also true if the image of $H$ in $G$ controls $\ell$-fusion
 in $G$ (e.g., if $K = 1$).
  \end{prop}

\begin{proof}
Since $K$ and $Z$ have orders that are prime to $\ell$, 
the map $H \to G$ induces a bijection of $\ell$-subgroups.
Furthermore, conjugacy in $H$ implies conjugacy in $G$, with the
converse also being true if the image of $H$ in $G$ controls
$\ell$-fusion in $G$. Note that the image of $H$ in $G$ is isomorphic
to $H/Z$.
The statement about torsion free ranks follows using the standard translation by Theorem~\ref{thm:poset2}.
\end{proof}

We conclude this section with a discussion of our conventions for
finite groups of Lie type. 

\begin{rem}[Finite groups of Lie type] \label{rem:lietype}
As stated in Definition~\ref{def:lietype} we take a finite group of Lie
type to mean a group of the form $G = \G^F$, for $\G$ a connected 
reductive algebraic group over an algebraically closed field of
positive characteristic $p$,
and $F$ a Steinberg endomorphism. We refer to \cite{MT}, or the
original \cite{steinberg68}, for a
thorough description of properties of such groups, but quickly go
through a few key points to aid to the reader: A connected reductive algebraic group
$\G$ over an algebraically closed field is classified by its root datum
$\D$ (which is field independent). The action of $F$ on $\G$ (up to inner automorphisms) is
also determined by its effect on $\D$ (up to Weyl group conjugation) allowing for a ``combinatorial'' classification
of finite groups of Lie type $\G^F$. It is most explicit when $\G$ is further
assumed simple, see \cite[Table~22.1]{MT}. In this case  $\G^F$ is
``close'' to being simple, in the following sense:
A formula of Steinberg
\cite[Corollary~12.6(b)]{steinberg68} says that $ G/O^{p'}(G)
\xrightcong \pi_1(\G)_F$, the coinvariants of the action of
$F$ on the fundamental group $\pi_1(\G)$. (As usual $O^{p'}(-)$
denotes the
smallest normal subgroup of $p'$ index, and $O_{p'}(-)$ denotes 
the largest normal subgroup of $p'$ order.)
Thus, subgroups $H$ with
$O^{p'}(G) \leq H \leq G$ can be parametrized by ``Lie theoretic''
data
consisting of $\G$, $F$, and a
subgroup of $\pi_1(\G)_F$. They are hence ``close'' to finite groups of
Lie type, though, e.g., the order formula
\cite[Corollary~24.6]{MT} does not hold --- some books dealing with finite
{\em simple} groups,  e.g., \cite[Definition~2.2.1]{GLS}, instead
refer to groups of the form $O^{p'}(\G^F)$ as finite groups of Lie type. Dual to $p'$-quotients we have that
\begin{equation}
 Z(G) =  O_{p'}(G) = Z(\G)^F
\end{equation}
(see \cite[Lemma~24.12]{MT}).
Normal $p'$--subgroups and quotients are related, as
$ 
   \G_{sc}^F/Z(\G_{sc}^F) \xrightcong O^{p'}((\G/Z(\G))^F),
$  
for  $\G_{sc}$ the simply connected
cover of $\G$ (see
\cite[Proposition~24.21]{MT}).  With a few small exceptions  \cite[Theorem~24.17]{MT},
this is a finite simple group, if $\G$ is simple.
For example $\PSL_n(q)\cong O^{p'}(\PGL_n(q))$ is simple unless
$(n,q)$ is $(2,2)$ or $(2,3)$.
We determine $\TF(H)$ for
for such groups $H$ in Section~\ref{sec:assoc}.
\end{rem}

\section{When ${\mathbb G}$ is simple, $3 \leq \ell \neq p$: Generic
  case}
  \label{sec:lg3generic} 

In this section $G$ is a finite group of Lie type as in
Definition~\ref{def:lietype}, where we furthermore assume that the 
ambient algebraic group $\G$ is
simple (and hence determined by an irreducible root system and an
isogeny type). The aim of Sections~\ref{sec:lg3generic} and ~\ref{sec:lg3specific} is to prove the following. 

\begin{thm}\label{thm:lgeq3} Let $G=\G^{F}$ be a finite group of Lie type where $\G$ is a simple algebraic group. Assume that
$3 \leq \ell \neq p$ and that $\rk_\ell(G) \geq 2$.
Then $\TF(G)\cong {\mathbb Z}$ except in the following cases:
\begin{itemize} 
\item[(a)] ${\ell }\geq 3$ and $G$ is isomorphic to either 
$\PGL_{\ell}(q)$ with $\ell$ dividing
$q-1$ or  $\PGU_{\ell}(q)$ with $\ell$ dividing $q+1$. 
In these cases,  
$\TF(G)\cong {\mathbb Z}\oplus  {\mathbb Z}\oplus  {\mathbb Z}$.
\item[(b)] ${\ell }=3$ and $G$ is isomorphic to $\ls3D_{4}(q)$. 
In this case,  $\TF(G)\cong {\mathbb Z}\oplus  {\mathbb Z}$.
\end{itemize} 
\end{thm}

The proof of Theorem~\ref{thm:lgeq3} entails a reduction, accomplished
in this section, to some 
cases of small rank and specific types.  
The analysis of the small rank cases 
is done in Section~\ref{sec:lg3specific}. 

The following is taken from \cite[Theorem 4.10.3]{GLS}.

\begin{thm} \label{T:uniquemax1}
Let $G=\G^F$ be a finite group of Lie type arising from a simple algebraic group $\G$ with a Steinberg endomorphism $F$,
and $\ell \neq p$, and write $\G  \cong \G_{sc}/Z$ for a finite central
subgroup $Z$. Assume that
\begin{itemize}
\item[(i)] the prime $\ell$ does not divide the order 
of $Z^F$. This is true if $\ell \nmid |Z(\G_{sc})^F|$.
\end{itemize}


\begin{itemize}
\item[(ii)] the prime $\ell$ is odd and good for $\G$ (meaning that
$\ell > 3$ if the type of $\G$ is $E_6$, $E_7$, $F_4$ or
$G_2$, $\ell >5$ if the type of $\G$ is $E_8$).
\end{itemize}
Then any elementary abelian $\ell$-subgroup $A$ of $G$ is contained in an
elementary abelian $\ell$-subgroup of maximal rank.
Also, any two elementary abelian $\ell$-subgroups of maximal
rank are conjugate except possibly if $\ell=3$
and $G \cong \ls3D_4(q)$.
\end{thm}

\begin{proof} Assume first that $\G$ is simply connected, i.e.,
  $Z$ is trivial.
Under condition (ii), \cite[Theorem 4.10.3(e)]{GLS} says that
 every elementary abelian $\ell$-subgroup of $G$ is contained
  in an elementary abelian $\ell$-subgroup of maximal rank. Finally \cite[Theorem 4.10.3(f)]{GLS}
implies that all maximal elementary abelian
  $\ell$-subgroups of $G$ are conjugate, unless $G \cong \ls3D_4(q)$, again  using (ii). This proves the theorem in the simply connected case.

Because $|Z^F|$ is assumed prime to $\ell$, 
the conclusion for $G$ follows from that of $G_{sc}$ by
Proposition~\ref{P:groupexactseq} applied the exact sequence 
\begin{equation} \label{eq:fourtermsequence}
\xymatrix{1\ar[r]&Z^F \ar[r]&G_{sc}\ar[r]&G\ar[r]&Z_F \ar[r]&1},
\end{equation}
of
\cite[Lemma~24.20]{MT},
where $|Z^F|=|Z_F|$ and $|G_{sc}|=|G|$
by \cite[Corollary~24.6]{MT}.
\end{proof}


The next proposition builds on Theorem~\ref{T:uniquemax1} and
handles many of the cases in Theorem~\ref{thm:lgeq3}, with the rest
being postponed to the next section. In the proof we employ the
non-standard notation, where e.g., $B_{2}(p)$ without subscript ``sc'' or ``ad'', denotes {\em any} group
arising from a simple algebraic group $\G$  over an algebraically
closed field of characteristic $p$ with root system
$B_2$, and $F = F_p$ is the standard Frobenius given by raising to the $p$th
power.

\begin{prop} \label{prop:oddreduce}
Let $\ell$ be an odd prime, $\ell \neq p$.
Suppose that $G=\mathbb G^{F}$ is a finite group of Lie type where ${\mathbb G}$ is a simple algebraic group and 
$F$ is a Steinberg endomorphism. Assume that the $\ell$-rank of $G$ is at
least $2$, and $G$ does not have one of
the forms $A_{n-1}(q)$ with $\ell$ dividing both $q-1$ and $n$,
$\ls2A_{n-1}(q)$ with $\ell$ dividing both $q+1$ and $n$, or
$\ls3D_4(q)$ with $\ell = 3$. Then $\TF(G) \cong \Z$.
\end{prop}

\begin{proof} Let $Z=Z(\G_{sc})$, whose order is given in \cite[Table
  9.2]{MT} (the order of ``$\Lambda(\Phi)$''). The order of $Z^F =
  Z(G_{sc})$ is given in  \cite[Table 24.2]{MT}. It follows from Theorem \ref{T:uniquemax1}  that $\TF(G)
  \cong \Z$ if $\ell$ is odd and good for $\G$, $\ell \nmid |Z^F|$, and $G$
  is not isomorphic to $\ls3D_4(q)$.
Consequently, it
remains to discuss the cases that either (i) $\ell$ divides $\vert Z^F \vert$, 
(ii)  $\ell = 3$ and $\G$ has exceptional type or (iii) $\ell = 5$ and
$\G$ has type $E_8$. We show, by explicit arguments, that in those cases
there are also no maximal
elementary abelian $\ell$-subgroups of rank
$2$, unless the $\ell$-rank of the group is two, in which case there
is a unique one. This shows that
$\TF(G) \cong \Z$ by Theorem~\ref{thm:poset2}.

First note that case (i) is basically ruled out by the
hyphotheses. That is, if
$\G$ has type $B_n, C_n$ or $D_n$,  then $\vert Z \vert$ is a
power of $2$ and hence is not divisible by $\ell$. If $\G$ has type
$A_{n-1}$ then the only cases where $\ell~|~\vert Z^F \vert$ are
exactly the ones we exclude in our formulation of the proposition.
Finally if $\G$ is of exceptional type and $\ell~|~|Z|$, then the
only possibility is $\G$ having type $E_6$ and $\ell = 3$, which is
covered under (ii) below.

This leaves (ii) and (iii), i.e., the exceptional types with $\ell=3$ and
$E_8$ with $\ell=5$. In other words, by the classification of
Steinberg endomorphisms \cite[Theorem~22.5]{MT}, the groups we need to consider are
$G_2(q)$,   $F_4(q)$, $\ls2F_4(q)$, $E_6(q)$,
$\ls2E_6(q)$, $E_7(q)$ and $E_8(q)$ at $\ell = 3$ and $E_8(q)$ at
$\ell =5$.  (Note that $\ls2F_4(q)$ only exists in characteristic $2$
and $\ls2G_{2}(q)$ does not appear on the list as we assume $q \neq 3$.)
We handle these groups on a case-by-case basis:

{\em $F_4(q)$, $E_6(q)$, $\ls2E_6(q)$, $E_7(q)$, and $E_8(q)$ with
  $\ell =3$: } We claim that in all these cases, there is an
elementary abelian $3$--subgroup of rank at least $4$, in fact inside a
maximal torus, which then shows  $\TF(G)\cong\Z$ by
Theorem~\ref{T:norank2}(a). When $\ell \nmid |Z^F|$ it is enough to see
that the multiplicity of the cyclotomic polynomials $\Phi_1$ and
$\Phi_2$ in the order polynomial of the complete root datum
$\ls{d}\,\D$ is (at least) $4$,  by
\cite[Theorem 4.10.3(b)]{GLS}. (Recall that a complete root datum
$\ls{d}\,\D$ consists of a root datum $\D$ together with the
twisting ``$d$'', see \cite[Definition~22.10]{MT} and \cite[Definition~2.2.4]{GLS}.)
This follows by
inspecting \cite[Part I, Table 10:2]{GL}. The only cases where
we can have
$\ell \mid |Z^F|$ are (again by  \cite[Table 24.2]{MT}) when either 
$E_6(q)$ with $q\equiv1\pmod3$ or $\ls2E_6(q)$ with $q\equiv-1\pmod3$.
But as the multiplicity of $\Phi_1$, respectively $\Phi_2$,  in the order polynomial of the
complete root datum $E_6$, respectively $\ls2E_6$, is $6$, we have
that the $\ell$-rank of $G_{sc}$ is (at least) 6 for these groups (again
by \cite[Theorem 4.10.3(b)]{GLS}), and hence the $\ell$-rank of $G$ is
at least $5$.

{\em  $G_{2}(q)$ with $\ell=3$:} We give a direct argument that all
elementary abelian $3$-subgroup of rank $2$ are conjugate.
By \cite[Lemma 4]{Az}, the commutator subgroup of the
centralizer of the center of a Sylow $3$-subgroup of $G$ is isomorphic
to $\SL_{3}(q)$ if $q\equiv1\pmod 3$, respectively to $\SU_{3}(q)$ if
$q\equiv-1\pmod 3$. 
In either case, any two elementary abelian $3$-subgroups of rank $2$
are conjugate by  Theorem~\ref{T:uniquemax1}.

{\em 
$\ls2F_4(2^{2a+1})$ with $\ell =3$:}
It follows from \cite[Proofs of (10-1) and (10-2), p. 118]{GL} that
$\ls2F_4(2^{2a+1})$ contains $\SU_3(2^{2a+1})$ of index prime to $3$.
All rank $2$ elementary abelian $3$-subgroups are conjugate in
$\SU_3(2^{2a+1})$ by Theorem~\ref{T:uniquemax1}, and hence this holds
for $\ls2F_4(2^{2a+1})$ as well.

{\em $E_8(q)$ with $\ell = 5$:}
From
\cite[Proofs of (10-1) and (10-2), p. 118]{GL} we see that
$E_8(q)$ contains $\SU_5(q^2)$ as a subgroup of index prime to $5$
(the coefficients are in $\mathbb{F}_{q^4}$). Hence, every elementary
abelian $5$-subgroup of $G$ is contained in one of rank $4$ by Theorem~\ref{T:uniquemax1}. 
Consequently, there are no maximal elementary abelian
$5$-subgroups of rank $2$. 
\end{proof}

\begin{rem} \label{rem:lcgsubgroups} For the interested reader, we briefly sketch how
  Proposition~\ref{prop:oddreduce} (and Theorem~\ref{T:uniquemax1})
  could alternatively be obtained via homotopy theory.  If $\ell$ does
  not divide the order of the fundamental group of a connected
  $\ell$-compact group $BG$, then every elementary abelian $\ell$--subgroup of rank at most $2$ is conjugate into a torus by
\cite[Theorem~1.8]{AGMV08}, generalizing Borel and Steinberg's classical theorem
\cite[Theorem~2.27]{steinberg75}.
The homotopical Lang square of Friedlander--Quillen \cite[(1)]{BM07} now
relates elementary abelian $\ell$--subgroups in $BG$ to those in the homotopical
finite group of Lie type $BG^{hF}$. When $F$ is the standard Frobenius
with $q$ congruent to $1$ modulo $\ell$ this shows that the
centralizer of every element of order $\ell$ in  $BG^{hF}$ has
$\ell$-rank at least the Lie rank of the $\ell$-compact group $BG$. For general $F$
one first uses untwisting \cite[Theorem~C.8]{GL20} to reduce to the
previous case, now inside another $\ell$-compact group. Note that untwisting assumes that the order of the
twisting is prime to $\ell$, which explains why
${}^3D_4(q)$ when $\ell =3$ needs to be treated separately. Indeed
the conclusion that $TF(G)$ has rank two in this case shows that this
is not only a technical limitation.
  \end{rem}


\section{When ${\mathbb G}$ is simple, 
$3 \leq \ell \neq p$: Specific cases} 
\label{sec:lg3specific}
In this section, we examine the cases not covered by 
Proposition~\ref{prop:oddreduce}, thereby completing the proof 
of Theorem \ref{thm:lgeq3}. The analysis is case by case, and we
assume $\ell \neq p$ throughout.

\begin{proof}[Proof of Theorem~\ref{thm:lgeq3}]  
First consider $G=\ls3D_4(q)$, with $\ell=3\nmid q$.
By \cite[Part I, 10-1(4)]{GL}, a Sylow $3$-subgroup $S$
of $G$ has the form $(C_{3^{a+1}}\times C_{3^a})\rtimes C_3$,
where $3^a=|q^2-1|_3$. From \cite[Theorem 5.10]{DRV},
we also know that $S\cong B(3,2(a+1);0,0,0)$ is a $3$-group
of maximal nilpotency class of $3$-rank $2$ and order $3^{2a+2}$.
Let $A$ be the maximal subgroup of $S$ of the form
$C_{3^{a+1}}\times C_{3^a}$, let $B$ be the subgroup of $A$
formed by the elements of order $3$, and let $V_1$ be any
non-normal maximal elementary abelian subgroup of $S$
(necessarily of rank $2$). The subgroups $B$ and $V_1$ are
those denoted likewise in \cite{DRV}. In \cite[Theorem 5.10]{DRV},
the authors prove that all the non-normal maximal elementary
abelian subgroups of $S$ are $G$-conjugate and that $V_1$ is
$\mathcal F$-Alperin. This means that $V_1$
and $N_S(V_1)$ are Sylow $3$-subgroups of $C_G(V_1)$ and $N_G(V_1)$,
respectively, and that $N_S(V_1)/V_1$ has no nontrivial
normal $3$-subgroup. From the description of $S$, it is clear
that $B$ is not a Sylow $3$-subgroup of $C_G(B)$ (hence
not $\mathcal F$-Alperin), and therefore $B$ and $V_1$ cannot be
$G$-conjugate. Therefore, $\TF(G) \cong {\mathbb Z}\oplus {\mathbb Z}$. 

For the remainder of the proof assume that $G$ has type either
$A_{n-1}(q)$ with $\ell \geq 3$  and $\ell \mid q-1$ or
${}^2A_{n-1}(q)$ with $\ell \geq 3$ and $\ell \mid q+1$. We assume also that 
$\ell$ divides the order of $Z^F$ and thus $n$ is a multiple of $\ell$. 
If $n > \ell$, then $\TF(G) \cong {\mathbb Z}$ by
Theorem~\ref{T:norank2}(a).
Thus we are 
reduced to consider the cases $G=A_{\ell-1}(q)$ with
$q\equiv1\pmod{\ell}$, and $G=\ls2A_{\ell-1}(q)$ with $q\equiv-1\pmod{\ell}$.
Because $\ell$ is prime there are exactly two distinct isogeny types. 
If $\G$ is simply connected, the asserted result follows by Theorem~\ref{T:uniquemax1}.
We are left with the cases $G=\PGL_{\ell}(q)$ 
and $G=\PGU_{\ell}(q)$ with the appropriate congruences of $q$ modulo
$\ell$.
Because the $\ell$-local structures of the 
two groups are almost identical, we consider only $G=\PGL_{\ell}(q)$.

Let $\widehat{G} = \GL_{\ell}(q)$ with $\ell$ dividing $q-1$. 
We choose a Sylow $\ell$-subgroup of $\widehat{G}$ 
to be a subgroup of the normalizer of 
a maximal torus of diagonal matrices (see Theorem 
\ref{T:uniquemax1}). 
The normalizer of the torus is a
wreath product, of the form 
$N\cong\GL_{1}(q)^{\times \ell} \rtimes {\mathfrak S}_\ell$, where
${\mathfrak S}_\ell$ is the symmetic group on $\ell$ letters.
That is, it is the subgroup of diagonal matrices
with an action by the group of permutation matrices.
Let $\zeta$ be a primitive $\ell^{th}$ root of unity
in ${\mathbb F}_q$. Let $\gamma$ be a generator for the Sylow
$\ell$-subgroup of $\GL_{1}(q)$, so that
$\zeta = \gamma^{\ell^{s-1}}$ for some $s$ and $\gamma^{\ell^s}=1$.
Let $x$ be the $\ell\times\ell$ permutation matrix 
\[
x = \begin{bmatrix} 0 & 1 & 0 & \dots & 0 \\
0 & 0 & 1 & \dots & 0 \\
\vdots &\vdots & \vdots &\ddots & \vdots \\
 0 & 0 & 0 &\dots & 1\\
1 & 0 & 0 & \dots  & 0 \end{bmatrix},
\]
let $y$ be the diagonal matrix (of size $\ell$) with diagonal entries
$\gamma, 1, \dots, 1$,
and let $z = \gamma I$ be the scalar matrix. A Sylow
$\ell$-subgroup $\widehat{S}$ of $\widehat{G}$ is generated by $x$ and $y$.
Then a Sylow $\ell$-subgroup of $G$ is
$S\cong\widehat{S}/ \langle z \rangle$. 
The subgroup $\widehat{S}$ has a maximal
subgroup $T = \langle y, xyx^{-1}, \dots, x^{\ell-1}yx^{1-\ell} \rangle$, 
which is abelian.

Let $\phi: \widehat{S} \to S$ be the quotient map. We note that two
subgroups $E$ and $F$ in $S$ are conjugate in $G$ if and only if
their inverse images $\phi^{-1}(E)$ and $\phi^{-1}(F)$ are
conjugate in $\widehat{G}$. Consequently, to find the maximal elementary
abelian subgroups of rank $2$ in $S$, it suffices to look
for the subgroups $E$ of order $\ell^{s+2}$ in $\widehat{S}$ that contain
$z$ and have the property that $E/\langle z\rangle$ is elementary
abelian. For the sake of this proof, 
call such a group $Q2$-{\it elementary}.

For our analysis, we identify three subgroups. Let
$a = y^{\ell^{s-1}}$ and let $b$ be the diagonal matrix with
diagonal entries $1, \zeta, \zeta^2, \dots, \zeta^{\ell-1}$.
Notice that $xbx^{-1}b^{-1} = \zeta \cdot I = z^{\ell^{s-1}}$. Let
\[
E_1 = \langle a, xax^{-1}, \dots, x^{\ell-1}ax^{1-\ell}, z \rangle, \qquad
E_2 = \langle x, b, z \rangle, \qquad \text{and} \quad
E_3 = \langle ax, b, z \rangle.
\]
We claim that every $Q2$-elementary subgroup of $\widehat{S}$ is either
conjugate to one of $E_2$ or $E_3$ or is conjugate to a
subgroup of $E_1$. Note that $E_1$ is abelian whereas the
other two are not. Also, every element of order $\ell$ in
$E_2$ has determinant $1$, but this is not true of $E_3$. Hence,
$E_2$ and $E_3$ are not conjugate, and neither is conjugate
to a subgroup of $E_1$.

Note first that any $Q2$-elementary subgroup of $T$ must be
contained in $E_1$ as $E_1$ is a direct product of $\ell$
cyclic subgroups of order $\ell$ and $\langle z \rangle$
is a direct factor. In particular, $E_1/\langle z \rangle$ 
contains all elements of order $\ell$ in $T/\langle z \rangle.$
Suppose that $H$ is a $Q2$-elementary
subgroup that is not in $T$. Then $H$ contains an element of the
form $tx$ for some $t \in T$. By a direct calculation,
we notice that the centralizer
in $T/\langle z \rangle$ of the class of $x$ is a 
direct factor of $T/\langle z \rangle$ 
that is cyclic of order $\ell^s$. It is
generated by the image in $T/\langle z \rangle$ of diagonal matrix $u$ with entries $1, \gamma,
\dots, \gamma^{\ell-1}$.  The subgroup of elements of  order $\ell$ in
this group is generated by $b = u^{\ell^{s-1}}$. So we can assume that
$H = \langle tx, b, z \rangle$.

It remains to find the conjugacy classes.
Suppose that $w \in T$ is diagonal with entries $w_1,\dots, w_\ell$.
Then $wxw^{-1} = vx$ where $v$ has diagonal entries $w_1w_2^{-1},
w_2w_3^{-1}, \dots, w_\ell w_1^{-1}$. In other words, $x$ is conjugate
in $\widehat{S}$ to $vx$ for $v$ any 
diagonal matrix with entries $v_1, \dots,
v_\ell$ satisfying the condition that the product
$v_1 \cdots v_\ell = 1$. It follows that any possible conjugacy
class of $Q2$-elementary subgroups not in $T$
has a representative of the form
$H= \langle a^i x, b, z \rangle$ for $i = 1, \dots, \ell^s-1$. Now, 
$(a^i x)^\ell = z^i$. 
If $i = m\ell$ for some $m\geq 1$, then $v = a^ixz^{-m}$
has the property that $v^\ell = 1$. In this case $v = tx$ where
$t \in T$ has the property that the product of its (diagonal)
entries is $1$. Thus, $v$ is conjugate to $x$ by an element in
$T$, and $H$ is conjugate to $\langle x, b, z \rangle$.

So we are down to the situation that $H = \langle a^ix, b, z \rangle$,
for $i= 0, 1, \dots, \ell-1$. But now notice that $x$ is conjugate
to $x^j$ for $j = 1, \dots, \ell -1$ by a permutation matrix, an
$\ell$-cycle, that centralizes $a$ and normalizes $\langle b, z
\rangle$. It follows that if $i \neq 0$, then $a^ix$ is conjugate
to $a^ix^{-i}$ and $H = \langle a^ix, b, z \rangle$ is conjugate
to $E_3$. This proves the claim.

Recall that $E_1/\langle z \rangle$ 
has $\ell$-rank $\ell\geq3$. It follows 
that $E_1/\langle z\rangle$, $E_2/\langle z\rangle$ and
$E_3/\langle z\rangle$ are in three distinct connected components of
the orbit poset $\CA^{\geq 2}_\ell(G)/G$ of noncyclic 
elementary abelian $\ell$-subgroups and that 
there are no other components containing 
subgroups of rank $2$. In other words, $\TF(G)$ has rank~$3$. 
\end{proof}

We now establish the rank of $\TF(G)$ in 
some specific cases that are useful in Section~\ref{sec:assoc}.

\begin{prop} \label{prop:psl}
Suppose that  $\ell \geq 3$, and either 
$G \cong \PSL_{\ell}(q)$ with $q\equiv1\pmod{\ell}$,
or $G \cong \PSU_{\ell}(q)$ with $q\equiv-1\pmod{\ell}.$
Assume that if $\ell =3$, then $q\equiv1\pmod{9}$ in the first case 
and $q\equiv -1\pmod{9}$ in the second. 
Then $\TF(G)$ has rank $\ell +1$.
\end{prop}

\begin{proof}
The $\ell$-local structures of $\PSL_\ell(q)$ with $\ell$ dividing $q-1$
and $\PSU_\ell(q)$ with $\ell$ dividing $q+1$ are very similar. We
give the proof only in the case that $G = \PSL_\ell(q)$. The proof
in the case of $\PSU_\ell(q)$ follows by the same line of reasoning.

We continue mostly with the notation introduced in the proof of
Theorem~\ref{thm:lgeq3} for $G=A_{\ell-1}(q)$, except that
we let $H = \SL_\ell(q)$ and $G = \PSL_\ell(q) = H/\langle z \rangle$
where $z = \zeta I$ generates the center of $H$ (not the same $z$ as in the   
previous proof). A Sylow $\ell$-subgroup
of $H$ has the form $S = T \rtimes \langle x \rangle$, where
$T$ is the collection of diagonal $\ell$-elements having determinant $1$.
Any element of $S$ that is not in $T$ is a power of
an element of the form $ax$ for some $a \in T$.
We note that the diagonal element $y$ as above, 
with entries $\gamma, 1, \dots, 1$, is not in $H$. 
The subgroup $S$ is generated by $x$ and $w = x^{-1}y^{-1}xy$ which is 
diagonal with entries $\gamma, \gamma^{-1}, 1, \dots, 1$, and $T$ is generated
by the conjugates of $w$ by powers of $x$.  

A $Q2$-elementary subgroup, if
it is not contained in $T$, must have the form
$J_a = \langle ax, b, z \rangle$ for
some $a$ in $T$.  That is, these are the nonabelian subgroups $J$ 
such that $J/\langle z \rangle$ is elementary abelian of rank $2$.
Note that $J_a = J_{a^\prime}$
if and only if $a^{\prime}a^{-1} \in \langle b,z \rangle$.
So there are $\vert T \vert/\ell^2$ such subgroups. A direct
calculation shows that $N_S(J_a)$ has order $\vert S \vert/\ell^4$.
Thus there are exactly $\ell$ $S$-conjugacy classes of such subgroups.
Let $E_i = \langle w^ix, b, z \rangle$, for
$i = 0, \dots, \ell-1$. All of these subgroups are conjugate in 
$\wh G = \GL_\ell(q)$ by some
power of the element $y$. Our purpose is to show, however, that
no two of them are conjugate in $H$. The theorem then follows,
because our observation implies that the classes
$E_i/\langle z\rangle$ for $0 \leq i < \ell$ are distinct conjugacy
classes of maximal elementary abelian $\ell$-subgroups of
$\PSL_\ell(q)$ of rank $2$.
The subgroup $T/\langle z \rangle$
also has a maximal elementary abelian subgroup $E/\langle z \rangle$,
and none of the $E_i$'s is conjugate to a subgroup of $E$ since
the latter is abelian.

Consider the subgroup $N = N_H(E_0)$, the normalizer in $\SL_\ell(q)$ of
$E_0 = \langle x, b, z \rangle$. The subgroup $E_0$ is an extraspecial 
group of order $\ell^3$ and exponent $\ell$. Its outer automorphism group
is isomorphic to $\GL_2(\ell)$ (see the discussion in \cite{Win}).
Because the centralizer of $E_0$ in $H$ is the center of $H$, $N$ is
an extension
\[
\xymatrix{
1 \ar[r] & E_0 \ar[r] & N \ar[r] & U \ar[r] & 1}
\]
where $U$ is isomorphic to a subgroup of $\SL_2(\ell)$ since it must also 
centralize $\langle z \rangle$.

Observe that $E_0$ is a proper subgroup of $N_S(E_0)$. In particular,
there is an element $u$
of $T$ whose class generates the center of $S/\langle b, z \rangle$
that is in $N_S(E_0)$. Hence, $U$ has an element of order $\ell$. Moreover,
$N_H(T)/T$ is isomorphic to the symmetric group on $\ell$ letters. This
group has an $\ell-1$ cycle that normalizes the subgroup generated by the
class of the element $x$. It must also normalize $\langle b,z \rangle$
and $\langle u, b, z \rangle$. Consequently, $U$ contains the subgroup $B$
of upper triangular matrices in $\SL_2(\ell)$. Because $B$ is a maximal
subgroup of $\SL_2(\ell)$, we need only show that $U$ has at least one element
that is not in $B$ to conclude that $U \cong \SL_2(\ell)$.

Let $v$ be the Vandermonde matrix
\[
v = \begin{bmatrix} 1 & 1      & 1           &  \dots  & 1 \\
                    1 & \zeta   & \zeta^2     & \dots   & \zeta^{\ell-1} \\
                    1 & \zeta^2 & \zeta^4     & \dots   & \zeta^{2(\ell-1)} \\
               \vdots &  \vdots & \vdots    &  \dots  &   \vdots     \\
                    1 & \zeta^{\ell-1} & \zeta^{2(\ell-1)} &\dots& \zeta^{(\ell-1)^2}
       \end{bmatrix}
\qquad \text{so that} \qquad
v^2  = \begin{bmatrix} \ell & 0      &  \dots & 0        & 0 \\
                         0  & 0      &  \dots & 0        & \ell \\
                         0  & 0      & \dots & \ell     &  0    \\
             \vdots & \vdots    &  \dots  & &\vdots    \\
                         0 &  \ell  & \dots   & 0       & 0
       \end{bmatrix} .
\]
Note that the columns (and also the rows) are eigenvectors for the
matrix $x$ with corresponding eigenvalues $1, \zeta, \zeta^2,
\dots, \zeta^{\ell-1}$. Thus, we have that $xv = vb$.
The computation of the matrix $v^2$ is straightforward as
each row is orthogonal (under the usual dot product) to all but one
of the columns.

Next we note that the determinant of $v^2$ is 
$\varepsilon\ell^\ell = (\varepsilon\ell)^\ell$ where
$\varepsilon = \pm 1$, the sign depending on the parity of $(\ell-1)/2$. 
Because the group $\bF_q^\times$ is cyclic and $\ell$ is prime to $2$,
the determinant of $v$ is also an $\ell^{th}$-power. That is, there is some
$\mu$ in $\bF_q^\times$ such that $\Det(v) = \mu^\ell$
and $\mu^2 = \varepsilon \ell$. Let $h$ be
the product of $v$ with the scalar matrix $\mu^{-1}I$.
Then $\Det(h) = 1$, $h \in H$ and $xh = hb$.  In addition, $h^2$ has the same
form as $v^2$ except that the nonzero entries 
that are equal to $\ell$ in $v^2$
are replaced by $\varepsilon$ in $h^2$. 
That is, $h^2 = (1/\varepsilon\ell)v^2$. 
So we find that $h^2xh^{-2} = x^{-1}$ by direct calculation.
Also, we have that $h^{-1}xh = b$ and
$h^{-1}bh = x^{-1}$. So $h$ is in $N$ and its class in
$U$, identified in a subgroup of $\SL_2(\ell)$, is the matrix
\[
\begin{bmatrix} 0 & 1 \\ -1 & 0 \end{bmatrix}.
\]
This element is not in the subgroup $B$, and hence we have shown
that $U \cong \SL_2(\ell)$.

Because $N_H(E_0)/E_0$ is the outer automorphism group of $E_0$
we have that $N_{\widehat{G}}(E_0) = N_H(E_0)\widehat{Z}$,
where $\widehat{Z}$ denotes the center of $\widehat{G} = \GL_\ell(q)$.
The same holds if we replace $E_0$ by $E_i$ since they are conjugate
in $\widehat{G}$.  Thus, we have that
if $g \in N_{\widehat{G}}(E_i)$, then the determinant of $g$ is an
$\ell^{th}$ power of some element in $\bF_q^\times$.

Finally, suppose that there is an element $g$ in $H$ such that $gE_ig^{-1}
= E_j$ for $i< j$. We know also that $y^{j-i}E_i y^{i-j} = E_j$.
Therefore, $y^{i-j}g \in N_{\widehat{G}}(E_i)$. However, this is
not possible. The reason is that $\gamma$ is a 
generator of the Sylow $\ell$-subgroup
of the multiplicative group $\bF_q^\times$ and $0 < i-j < \ell$,
the determinant of $y^{i-j}g$, which is equal to
$\gamma^{i-j}$, is not an $\ell^{th}$ power. Hence, if $i \neq j$,
then $E_i$ is not $H$-conjugate to $E_j$ and 
then $E_i/\langle z \rangle$ is not $G$-conjugate to $E_j/\langle z \rangle$.
This proves the proposition.
\end{proof}

\section{Groups associated to finite groups of Lie type for $\ell\geq 3$}
\label{sec:assoc}

In this section we are interested in some of the  groups associated 
to finite groups of Lie type.
Suppose that $G_0 = G_{sc}$ is a finite group of Lie type
arising from a simply connected simple algebraic group $\G$.
If $G_0=\SL_n(q)$ or $\SU_n(q)$, let $G_1=\GL_n(q)$, or $\GU_n(q)$,
respectively. If $\G$ is symplectic or orthogonal, take $G_1$ to be
the conformal group of that type (cf. \cite[page 7,8]{MT} and
\cite[Section 2.7]{GLS}). For example, if $G_0 =\Sp_{2n}(q)$,
then $G_1 =\operatorname{CSp}_{2n}(q)$, the group of all $2n \times 2n$-matrices $X$
with the property that $XfX^{tr} = af$ for some $a \in \bF_q$, where 
$f$ it the matrix of the symplectic form. 
If $\G$ has exceptional type, let $G_1=G_0$.
By an {\em associated group} of
$G_0$ we mean a group $G=H/J$, where $G_0\leq H\leq G_1$ and
$J\leq Z(H)\leq Z(G_1)$ such that $G$ contains the group
$G_0/Z(G_0)$ as a section. For example,
in type $A_{n-1}$, an associated group is a quotient
$G = H/J$ where $\SL_n(q) \leq H \leq \GL_n(q)$ and
$J \leq Z(H) \leq Z(\GL_n(q))$. The simple group
$\PSL_n(q)$ is an example. In any type, a diagram for such groups
has the form
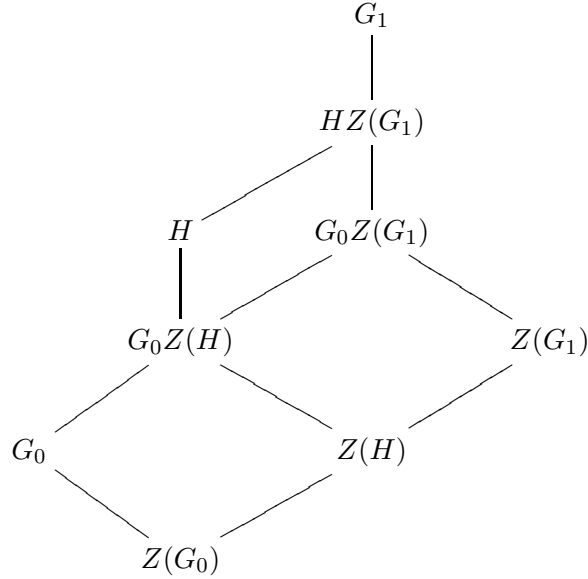
\[
\xymatrix{&&G_1\ar@{-}[d]\\&&HZ(G_1)\ar@{-}[d]\ar@{-}[dl]\\
&H\ar@{-}[d]&G_0Z(G_1)\ar@{-}[dl]\ar@{-}[dr]\\
&G_0Z(H)\ar@{-}[dl]\ar@{-}[dr]&&Z(G_1)\ar@{-}[dl]\\
G_0\ar@{-}[dr]&&Z(H)\ar@{-}[dl]\\
&Z(G_0)}
\]
where the associated group is $G = H/J$ for $J$ some subgroup of $Z(H)$.
Note that $J$ may or may not contain $Z(G_0)$.

Our analysis will entail understanding the structure of $G$, and will benefit substantially from knowing when 
$G$ is isomorphic to a product of groups. 

\begin{lemma} \label{lem:relprime}
In addition to the above notation, assume that $G_0 = [G_1,G_1]$.
Let $G = H/J$ be a section of $G_1$ as above. Write 
$H/[H,H] = U \times V$ where $U$ is a $\vert Z(G_0) \vert$-Hall
subgroup of $G_1/G_0$ and $\vert V \vert$ 
is relatively prime to $\vert Z(G_0) \vert$. Then there exist 
subgroups $H^\prime \leq H$ and $U^\prime \leq Z(H^\prime)$ such that 
the order of both $H^\prime/[H^\prime,H^\prime]$ and $U^\prime$ 
divide $\vert Z(G_0) \vert^n$ for some $n$
and $G \cong H^\prime/U^\prime \times V$.
\end{lemma}

\begin{proof}
Write $G_1/G_0 \cong U_1 \times V_1$ 
and $Z(G_1) \cong U_2 \times V_2$ where 
$ U_i$ is a $\vert Z(G_0) \vert$-Hall subgroup 
and $\vert V_i \vert$ is relatively prime to $\vert Z(G_0) \vert$ 
for $i = 1,2$. Consider the composition $\phi: Z(G_1) \to G_1 \to 
G_1/G_0$. By the construction $V_2 \cap G_0 = \{1\}$ so that $\phi$
is injective on $V_2$. On the other hand, $G_1/Z(G_1)$ is the finite
group arising as the fixed points of the simple algebraic group of adjoint isogeny type and we know that 
$\vert G_1/G_0Z(G_1) \vert = \vert Z(G_0) \vert$ which is relatively
prime to the order of $V_1$. It follows that $V_1 = \phi(V_2)$.


Because $G_0 \leq H$, we have that $G_0 = [H,H]$ and that 
$H$ is determined uniquely by its image in $G_1/G_0$.
Thus $H/G_0$ has the form $\tilde{U} \times \tilde{V}$, 
where $\tilde{U} \leq U_1$ and $\tilde{V} \leq V_1$.
Likewise, $J = \hat{U} \times \hat{V}$ and $Z(H) \cong U^\prime
\times V^\prime$ with 
$\hat{U}, U^\prime \leq U_2$ and $\hat{V}, V^\prime \leq V_2$. 
It follows that $V^\prime/\hat{V} \cong \tilde{V}/\phi(\hat{V}) \cong V$
is a direct factor of $H/J$. The statements about the group $H^\prime$ 
can be deduced from the fact that its composition factors are 
$G_0/Z(G_0)$, $U$ and $U^\prime/\hat{U}$.
\end{proof}

The main aim of the section is to prove the following theorem.

\begin{thm} \label{thm:assoc} Let $G_{0} = \G^F$ be a finite group of Lie
type, where $\G$ is simple and simply connected. Let $G$ be one of
the associated finite groups of $G_{0}$. 
Assume that $\ell \geq 3$ does not divide $p$ and that the $\ell$-rank of $G$
is at least $2$. Then $\TF(G) \cong \Z$ except in the following cases.
\begin{itemize}
\item[(a)] If $G \cong \PSL_{\ell}(q)$ with
  $q\equiv1\pmod{\ell}$ if $\ell >3$,
and with $q\equiv1\pmod9$ if $\ell = 3$, then $\TF(G)$ has rank $\ell +1$.
\item[(b)] If $G \cong \PSU_{\ell}(q)$ with
  $q\equiv-1\pmod{\ell}$ if $\ell >3$,
and with $q\equiv-1\pmod9$ if $\ell = 3$, then $\TF(G)$ has rank $\ell +1$.
\item[(c)] If ${\ell }=3$ and $G \cong \ls3D_{4}(q)$,
then $\TF(G)$ has rank 2.
\end{itemize}
\end{thm}

\begin{proof}
The last case (c) was treated in Section \ref{sec:lg3specific}
(see also Theorem \ref{thm:lgeq3}).

Assume that the group has the form $G = H/J$ as in the previous notation
of the section. We prove the theorem for groups of Lie type $B_n$, $C_n$,
$D_n$ and $\ls2D_n$, by noticing that $G_0=G_{sc}$ 
has center that has order either $2$ 
or $4$ (see \cite[Table 24.2]{MT}). 
Consequently, if $\ell$ divides the order of $Z(G) = Z(H)/J$ then $G$
has a direct factor that is a cyclic $\ell$-group. In such a case the 
center of a Sylow $\ell$-subgroup of $G$ has $\ell$-rank at least $2$ and
we are done. On the other hand, if $\ell$ does not divide the 
order of $Z(G)$, then by Lemma \ref{lem:relprime}, a Sylow $\ell$-subgroup
of $G$ is isomorphic to that of $G_0$. These cases have already been 
considered. 

A similar thing happens in types $A_n$ and $\ls2A_n$. That is, if $\ell$
does not divide the order of $Z(G_0)$, then regardless of whether $\ell$
divides $|Z(G)|$, we are done by the same arguments as above. Consequently,
we can assume that $\ell$ divides the order of $Z(G_0)$, requiring that $\ell$ 
divides both $n+1$ and $q-1$ in type $A_n$, and that $\ell$ 
divides both $n+1$ and $q+1$ in type $\ls2A_n$.

For the untwisted type $A_n$, we need to consider the case when $\ell$
divides both $n+1$ and $q-1$. However, if $n+1 > \ell$, the
$\ell$-rank of $G$ is greater than $\ell$, and therefore $G$ cannot
have any maximal elementary abelian $\ell$-subgroup of rank $2$.
So it remains to consider the case $\ell = n+1$ with
$q\equiv1\pmod{\ell}$.
Similarly, in the twisted case $\ls2A_n$, we may assume that
$\ell = n+1$ with $q\equiv-1\pmod{\ell}$. In addition, by Lemma 
\ref{lem:relprime}, we may assume that the orders of $J$ and $H/G_0$ 
are powers of $\ell$. 

If $J = \{1\}$, then $G \leq \GL_\ell(q)$ or $G \leq \GU_\ell(q)$. In
either case, an eigenvalue argument tells us that any element of order
$\ell$ is conjugate to an element of the diagonal torus. Hence, 
we are done in this case, and we may assume that $J\neq \{1\} $. 

If $J \neq Z(H)$, then there exists an element $x$ in $Z(H)$ such 
that $x \notin J$ but $x^\ell \in J$. Also, because $J$ is not trivial, 
there exists an element of order $\ell$ in the diagonal torus in $H$
whose class in $H/J$ is central in a Sylow $\ell$-subgroup. Thus, in 
such a case, the center of a Sylow $\ell$-subgroup of $H/J$ has 
$\ell$-rank $2$ and we are done by Lemma \ref{lem:normal3rk}. So assume
that $J = Z(H)$. 

In the untwisted situation, we are down to two possibilities. First if
$H/G_0$ is a Sylow $\ell$-subgroup of $G_1/G_0$ then $J$ is a 
Sylow $\ell$-subgroup of $Z(G_1)$. In such a case $G = H/J \cong \PGL_\ell(q)$. 
This case has been treated in Section~\ref{sec:lg3specific}. In the other case,
that $J < Z(G_1)$, we have that $G \cong \PSL_\ell(q)$ and $\ell$ 
divides $q-1$. 
Similarly, in the twisted case we are down to the situation that 
$G \cong \PSU_\ell(q)$ and $\ell$ divides $q+1$.

Observe that if $\ell = 3$, with 3 dividing $q-1$ and 9 not dividing
$q-1$, then a Sylow $3$-subgroup of $\PSL_3(q)$ is elementary abelian of
order $9$. The same holds for $\PSU_3(q)$ if $3$ divides $q+1$ and
$9$ does not divide $q+1$. Hence, $\TF(G)$ has rank $1$ in both of these
cases.  Thus, it remains to
calculate the ranks of $\TF(G)$ in the cases (a) and (b) of
the theorem. This is accomplished in Proposition \ref{prop:psl}.
\end{proof}


\section{When ${\mathbb G}$ is simple, $2 = \ell \neq p$}
\label{sec:le2generic}

The goal of this section is to establish Theorems~\ref{thm:l=2} and~\ref{thm:char2}. Some results
of this section will also be used in Section  \ref{sec:reductive}.

\begin{thm}\label{thm:l=2}
Let $G$ be a finite group of Lie type (see
Definition~\ref{def:lietype}) with the ambient group $\G$ a simple
algebraic group. Suppose $\ell=2 \neq p$ and that $TF(G)$ has rank
greater than one. Then  $G$ has nonabelian dihedral Sylow
$2$-subgroups, $G \cong \PGL_{2}(q) \cong \PGU_{2}(q)$ for $q$ odd, and $TF(G)
\cong \Z \oplus \Z$
\end{thm}

We also calculate the ranks of $TF(G)$ when $G$ is 
one of the associated groups in the case that $\ell = 2$
is not the defining characteristic of the group. 
The notion of an associated group was introduced in 
Section \ref{sec:assoc}. We adopt the notation used at the
beginning of Section \ref{sec:assoc}. In particular,
$G_1$ is one of the general linear or conformal group such as
$\GL_n(q)$, $\GU_n(q)$ or $\operatorname{CSp}_n(q)$ and $G_0=G_{sc}$.
The group $G =H/J$ is a section of $G_1$ such
that $G_0 \leq H \leq G_1$ and $J \leq Z(H)$.

The groups of endotrivial modules for the associated groups of type 
$A_n$ are determined in the paper \cite{CMN4}. Our aim in this section
is to take a more conceptual and less technical approach. 
For this reason some arguments from \cite{CMN4} are included here. 
In particular, exceptional cases occur when $G_0 \cong \SL_2(q)$,
and some additional explanation is provided.

\begin{lemma} \label{lem:sl2} ({\it cf.} \cite[Lemma 10.1]{CMN4})
Suppose that $G = H/J$ is a group associated to $\SL_2(q)$. That is,
$\SL_2(q) \leq H \leq \GL_2(q)$ and $J \leq Z(H)$.
Write $Z(H)/J = UV$ where $U$ is a $2$-group and $V$ has odd order, and
let $W$ be the kernel of the quotient followed by  projection $Z(H) \to Z(H)/J\to V$.
Then $H/J \cong H/W \times V$ and $\TF(G)\cong \TF(H/W)$. Moreover, if
$U = \{1\}$, meaning that $Z(H)/J$ has odd order, 
then $H/W$ is isomorphic to either $\PSL_2(q)$ or
$\PGL_2(q)$. An analogous situation occurs in the case that $G$ is
an associated group of $\SU_2(q)$.
\end{lemma}
\begin{proof}
The proofs are similar in both types. Thus we only prove the
statement for $G$ associated to $\SL_2(q)$. 
We see that $Z(H)$ is a subgroup of the group of nonzero scalar matrices.
If $\psi: \bF_q^\times \to \GL_2(q)$ is
the map that takes $\zeta$ to $\zeta Id$,
then its composition with the determinant map is the squaring operation.
Hence, it is an isomorphism on odd order 
subgroups of $\bF_q^\times$. The first
statement follows by this argument. The second statement is a consequence
of the Jordan-H\"older Theorem and the observation that $H/W$ has the same
order as $\PSL_2(q)$ except if $H= \GL_2(q)$ in which
case $H/W\cong \PGL_2(q)$.
\end{proof}

Then our main theorem to address the associated groups is the following. 

\begin{thm} \label{thm:char2} Let $G \cong H/J$ be
an associated group of a finite group of Lie type as defined above
with $q$ odd, and let $\ell=2$.
Then $\TF(G) \cong \Z$ is cyclic except in the following cases.
\begin{itemize}
\item[(a)] $G = \SL_{2}(q)\cong\SU_{2}(q)$. 
\item[(b)] $G = \PSL_{2}(q) \times C \cong \PSU_{2}(q) \times C$
with $q \equiv\pm1\pmod8$ and $C$ a cyclic group of odd order. (See Lemma
\ref{lem:sl2}.)
\item[(c)] $G = \PGL_{2}(q) \times C \cong\PGU_{2}(q) \times C$,
where  $C$ is a cyclic group of odd order.
\end{itemize}
In case (a), a Sylow $2$-subgroup of $G$ is quaternion and $TF(G)=\{0\}$.
In cases (b) and (c), $Z(H)/J$ has odd order, a Sylow $2$-subgroup of
$G$ is (nonabelian) dihedral and $\TF(G) \cong \Z \oplus \Z$.
\end{thm}

In the proof, we first show that the theorem holds for groups of
large Lie rank.
The groups of small Lie rank are considered on a
case by case inspection. The main reduction theorem is taken from 
\cite{GH}.

\begin{thm} \label{thm:largerank} Let $\wh G = \G^F$ be a finite 
group of Lie type in odd
characteristic, with ${\mathbb G}$ simple and simply
connected, and set $\ell =2$. 
Then $\TF({G})\cong {\mathbb Z}$, for ${G}$ any
associated group to $\wh G$, as defined above, provided that $\wh G$ is not one of
the following types.
\begin{itemize}
    \item[{}]
    \begin{AutoMultiColItemize}
\item[(a)] $\,\,A_1(q)$, $A_2(q)$, $\ls2A_2(q)$,
\item[(b)] $\,\,A_3(q)$ for  $q \not\equiv 1\pmod8$,
\item[(c)] $\,\,A_4(q)$  for $q \equiv -1\pmod4$,
\item[(d)] $\!\!\ls2A_3(q)$  for $q \not\equiv 7\pmod8$,
\item[(e)] $\!\!\ls2A_4(q)$  for $q \equiv 1\pmod4$,
\item[(f)] $\,\,B_2(q)$,   
\item[(g)] $\!\!\ls3D_4(q)$,
\item[(h)] $\,\,G_2(q)$, or $\ls2G_2(q).$
    \end{AutoMultiColItemize}
\end{itemize}
\end{thm}

\begin{proof} 
 Recall that by Tits' theorem \cite[Theorem~24.17]{MT}  $\wh G/Z(\wh G)$ is
  simple, except in a few cases which are among the cases excluded above.
  In  \cite[Main Theorem]{GH}, all finite simple groups
  having sectional $2$-rank at most $4$ are listed. 
If the finite simple group associated to $\wh G$ is not on the above
list, then $G$ has sectional $2$-rank greater than $4$.
(See \cite[Section~3.5]{atlas} or \cite[Theorem~2.2.10]{GLS} for
a list of isomorphisms between finite groups of Lie type.)
So ${G}$ has no maximal elementary abelian
 $2$-subgroups of rank $2$, by Theorem~\ref{T:norank2}(b) as desired.
\end{proof}

We may now complete the proofs of the main theorems of this section.
For the proof, recall that
if $G \cong A \times B$, with $B$ of order prime to 
$\ell$, then $\TF(G) \cong \TF(A)$,  by Proposition \ref{P:groupexactseq}.

\begin{proof}[Proof of Theorems \ref{thm:l=2} and \ref{thm:char2}] By Theorem
\ref{thm:largerank}, we need only deal with 
the groups listed. 
The Sylow $2$-subgroups of finite groups of Lie type are known to be
cyclic only when $G$ is associated to a finite group of Lie type
$A_1(2)$. The groups $\SL_2(q) \cong \SU_2(q)$ have quaternion Sylow 
$2$-subgroups, and hence $\TF(G) \cong \{0\}$ in those cases. 

Recall that for any finite group $G$ with
(nonabelian) dihedral Sylow $2$-subgroup we have
$\TF(G)\cong\Z\oplus\Z$ as it is not possible for
the two $S$-conjugacy classes of elementary abelian subgroups
of order $4$ in $S$
to fuse in $G$ (cf.\ \cite[Section 3.7]{Mazza}).
The Sylow $2$-subgroups of the groups in Theorem~\ref{thm:char2}(b)
are nonabelian dihedral. Note that if $q \equiv\pm3\pmod8$ then the
Sylow $2$-subgroups of $\PSL_2(q)$ are elementary abelian of order $4$, and
$\TF(\PSL_2(q))\cong \Z$. It is easily verified that 
the Sylow $2$-subgroups of $\PGL_2(q) \cong \PGU_2(q)$ are dihedral 
and not abelian. So $\TF(G) \cong \Z \oplus \Z$ in this case. 

An eigenvalue argument tells us that any involution in $H$ for either
$\SL_2(q) \leq H \leq \GL_2(q)$ or 
$\SU_2(q) \leq H \leq \GU_2(q)$ 
is conjugate to a diagonal matrix. In the unitary case, note that the 
eigenspaces of an involution are orthogonal to each other, so that 
we can construct a change of basis matrix that is unitary. 
Hence, $\TF(G) \cong \Z$ if
$J$ has odd order.  Therefore, for the proof for groups
of type $A_1$, we need only consider quotients ${G} = H/J$
where $J$ has even order.

Note that $\GL_2(q)$ is not isomorphic to $\GU_2(q)$. However, 
arguments for these cases are almost identical. 
That is, we can find 
$q'$ with $q'\equiv -q\pmod4$ such that $\SL_2(q')$ or
$\GL_2(q')$ have isomorphic Sylow $2$-subgroups to those of $\SU_2(q)$
or $\GU_2(q)$, respectively
(cf. \cite[Section 1]{CF}). So we prove only the linear case. 

If $q \equiv 3 \pmod 4$, then $4$ does not divide the order of $Z(\GL_2(q))$.
By our assumptions,  $Z(H)/J$ has odd order and by Lemma \ref{lem:sl2} 
we are done.
So we may assume that $q \equiv 1 \pmod 4$ and that $Z(H)/J$ has even order. 
Then there is an element 
$z$ in $Z(H)$ that represents a nontrivial involution in $H/J$. In addition,
the diagonal matrix with entries $1$ and $-1$ is an involution whose 
image in $H/J$ is central in a Sylow $2$-subgroup and distinct from
the image of $z$. Thus the center of a Sylow $2$-subgroup of $H/J$ has 
$2$-rank two and $\TF(H/J) \cong \Z$ by Lemma \ref{lem:normal3rk}.  

\smallskip
\noindent 
{\bf Types $A_2$, $A_4$, $\ls2A_2$ and $\ls2A_4$.}  
The proofs that $\TF(G) \cong \Z$ for groups types $A_2$ and $A_4$
are given in \cite[Sections 6 and 9]{CMN4}. The structure of the 
Sylow $2$-subgroups are very similar for the twisted and untwisted 
cases \cite{CF}. Hence, we leave the proofs of the twisted cases, 
$\ls2A_2$ and $\ls2A_4$, to the reader. We note that centers for 
all finite groups $G_{sc}$ of these types have odd order. Consequently, 
by the argument in
the proof of Lemma \ref{lem:sl2}, the Sylow $2$-subgroup of
$Z(H)/J$ of these types is a direct factor, 
which can be assumed to be trivial for the purposes of the proof.

\smallskip
\noindent 
{\bf Types $A_3$, $\ls2A_3$ and $B_2$.}
We prove the results only for groups of type $A_3$ and $B_2$, because 
the proofs for groups of type $\ls2A_3$ are very similar to those of 
type $A_3$ (in the $\ls2A_3$ case, we take the matrix of the hermitian
form to be the identity matrix). 
Following the notation introduced at the beginning of
Section~\ref{sec:le2generic}, let $G_0$ be $\SL_4(q)$ or 
$\Sp_4(q) \cong \Spin_{5}(q)$ in type $A_3$ or $B_2$, respectively. 
Let $G_1 = \GL_4(q)$ in the first case and  
$G_1 = \CSp_4(q)$ in the second. Here, $\CSp_{4}(q)$
is the group of $4\times 4$ matrices $X$ with
entries in $\bF_q$ having
the property that $X^{tr}fX = af$ for some $a \in \bF_q^{\times}$,
$f$ being the matrix of the symplectic form. For the purposes of this 
proof assume that the 
symplectic form is given as 
\[
f = \begin{bmatrix}  0& 1& 0& 0 \\ -1& 0& 0& 0 \\ 0& 0& 0& 1 \\ 0& 0& -1& 0
\end{bmatrix}.
\]

Let $G = H/J$ be a group associated to $G_0$.
That is, $G_0 \leq H \leq G_1$ and $J \leq Z(H)$.
Then a Sylow $2$-subgroup $S = S_{G}$ of $G$ is a section of a
Sylow $2$-subgroup $S_{G_1}$ of $G_1$. Indeed, a Sylow $2$-subgroup 
$S_H$ of $H$ is subgroup of a Sylow $2$-subgroup $R$ of $\GL_4(q)$. 
The group $R$ is isomorphic to a wreath product 
$R = (U_1 \times U_2) \rtimes C_2$ where $U_1, U_2$ are Sylow $2$-subgroups
of $\GL_2(q)$ \cite{CF}. In particular, we use the 
following notation:
\[ 
s(A,B) = \begin{bmatrix} A & 0 \\ 0 & B \end{bmatrix}, \qquad 
t(A,B) = \begin{bmatrix} 0 & B \\ A & 0 \end{bmatrix} = ws(A,B),
\]
where these are matrices of $2 \times 2$ blocks, $A$ and $B$ are elements
of $\GL_2(q)$ and $w = t(I,I)$. 
Then $R$ is generated by all
$s(A,B)$ for $A$ and $B$ in $S_{\GL_2(q)}$ and the element $t(I,I)$ where
$I$ is the $2\times 2$ identity matrix. Note that an element of $J$ must
be a scalar matrix $s(\zeta I, \zeta I)$ for some $J$. Because of the 
choices of the form, there are Sylow $2$-subgroups
of $\CSp_4(q)$ that respect this structure

Note that there exist subgroups $D_J$ and $M_H$ of 
$\bF_q^\times$ that determine $J$ and $H$. That is, $J$ is the 
set of all scalar matrices with diagonal entry in $D_J$. In type
$A_3$, $H$ is the subgroup of all elements in 
$\GL_4(q)$ with determinant in $M_H$. In type $B_2$, $H$ is the 
subgroup of all $X$ with $X^{tr}fX = af$ for some $a \in M_H$.

Suppose that $J$ has odd order.
Then, by an eigenvalue argument (cf. \cite[Lemma 3.3]{CMN3}), 
any involution
in $H$ is conjugate to a diagonal matrix. Note that in type $B_2$ (and 
$\ls2A_3$), the eigenspaces $V_1$ and $V_{-1}$ corresponding to the
eigenvalues $1$ and $-1$ of an involution $u$ are orthogonal to each other. 
Consequently, there exists a change of basis matrix that conjugates $u$
into a diagonal matrix and also preserves the form, and it is 
an element of $H$. It follows that every elementary abelian
$2$-subgroup in $G$ 
is conjugate to a subgroup of the image modulo $J$ of the 
group of diagonal elements of order $2$ in $H$.
Hence, in this case we are finished. For the rest of 
the proof assume that $J$ has even order.

Next suppose that $S_J \neq S_{Z(H)}$. That is, suppose that there is an 
element of the center of $H$ whose order is a power of $2$, and that is not 
in $J$. In particular there exists a scalar element of $H$ whose square 
is in $J$. In addition, because the order of $J$ is even, the element 
$s(I, -I)$ is central in $S = S_{G}$. Thus $Z(S)$ has 
$2$-rank $2$ and we are done by Lemma \ref{lem:normal3rk}.

We have reduced the proof to the situation in which $S_J = S_{Z(H)}$.
Our aim is to show that the centralizer of every involution in $S$
has $2$-rank at least $3$. This will complete the proof in the cases of types
$A_3$ and $B_2$ (and $\ls2A_3$). 

First consider involutions represented modulo $J$ by a matrix 
of the form $s(A,B)$ in the case that 
$q \equiv 1\pmod4$ and the type is $A_3$ or $B_2$. (The argument in the 
case or type $\ls2A_3$ with $q \equiv 3\pmod4$ is very similar.)
In this case, a Sylow $2$-subgroup of $\GL_2(q)$ is 
generated by the elements
\[
W = \begin{bmatrix} 0 & 1 \\ 1 & 0 \end{bmatrix},  \quad
Y = \begin{bmatrix} 0 & -1 \\ 1 & 0 \end{bmatrix} \quad
\text{  and } \quad 
X_\zeta = \begin{bmatrix} \zeta & 0 \\ 0 & 1 \end{bmatrix}
\] 
for $\zeta$ a generator of the Sylow $2$-subgroup of $\bF_q^\times$. Let $T$ be the subgroup of 
$S_{\GL_2(G)}$ generated by the scalar matrices of the form $WX_{\zeta^m}WX_{\zeta^m}$ for any $m$. 
If the class of $u = s(A,B) \in H$ is an involution in $H/J$, 
then $A^2 = B^2 = \mu I$ for some $\mu\in\F_q^\times$. The quotient 
$S_{\GL_2(q)}/T$ is a dihedral group generated by the classes of 
$W$ and $X_\zeta$. An involution in this group must be represented
by either $W$ or $X_{\zeta^m}$ for some $m$. Then if 
the class of $u = s(A,B)$ is an involution in $H/J$, 
it has either the form $s(X_{\zeta^m},X_{\zeta^m})$
or $s(A,B)$ with $A$ and $B$ in the subgroup $V = \langle X_{-1}, W \rangle$.
Now notice that the subgroup generated by $w$ and all $s(A,B)$ with 
$A, B \in V$ is elementary abelian of $2$-rank at least $3$. If 
$u = s(X_{\zeta^m},X_{\zeta^m})$ is in $H$, then so also is $w$ and
$s(I, -I)$, and the classes of these elements generate a subgroup of
$H/J$ having $2$-rank $3$. So we are done in this case. 

Next suppose that the class of $s(A,B)$ is an involution in $H/J$,
in the case that $q \equiv 3\pmod4$ and 
the type is $A_3$ or $B_2$.
(The same argument works when the type is $\ls2A_3$ with 
$q \equiv 1\pmod4$.) In this case $J = Z(\GL_4(q))$ has order $2$ and
is generated by $-I_4$, where $I_4$ is the $4 \times 4$ identity matrix.
A Sylow $2$-subgroup $S_{\GL_2(q)}$ is semidihedral. In this case
one of two things can happen. The first is that $A$ and $B$ are actual
involutions. If $A$ is a noncentral involution, the subgroup generated 
by the classes of $w$, $s(A,A)$ and $s(I, -I)$ has $2$-rank $3$
in $H/J$. The other possibility is that $A$ and $B$ have order $4$ and 
commute modulo $J$. The only possibility here is that $A$ and $B$
are contained in a quarternionic subgroup of order $8$ in 
$S_{\GL_2(q)}$. If $A$ is not contained in the subgroup generated by $B$
then the classes of $w$, $s(A,B)$, and $s(B,A)$ generate an elementary 
abelian subgroup in $H/J$ of order $8$. Otherwise, let $X$ be another generator
of the quaternionic subgroup. Then the classes of $w$, $s(A,B)$ and
$s(X,X)$ generate an elementary abelian subgroup of order $8$. So we
are done in this case. 

Finally, suppose that the class of $u = t(A,B)= ws(A,B)$ 
is an involution in $H/J$.
It must be that $AB = BA = \mu I$ for some $\mu \in \bF_q^\times$. 
That is, $B = \mu A^{-1}$. In the case that the type is $A_3$, then 
$s(A, I)^{-1}t(I, \mu I)s(A, I)= t(A,B)$.
So every such involution is conjugate to one of the form 
$y_{\mu} = t(I, \mu I)$. In turn, any $y_\mu$ commutes with any 
involution $s(A,A)$ for $A$ not central in $S_{\GL_2(q)}$. Thus, in
type $A_3$, the centralizer of $u$ has $2$-rank at least $3$, and we
are done.  

So suppose the type is $B_2$. We have that $ufu^{tr} = \mu f$ implying
that $AYA^{tr} = Y$, as expected. 
A set of representatives of the generators of 
$S_{\GL_2(q)}$ can be chosen so that their product with their
transpose is a scalar matrix (see the above descriptions in addition
to \cite{CF}). 
The implication is that $v = t(y,y)$ commutes with $u$. Thus, the 
centralizer of $u$ has $2$-rank at least $3$, as it contains
the image in $H/J$ of $\langle u, j, t(-I,I) \rangle$.

To summarize, we have proved that the centralizers of the involutions in a group associated to a finite group of Lie type $A_3$, $\ls2A_3$ and $B_2$ have $2$-rank at least $3$, and so there are no 
maximal elementary abelian $2$-subgroups of rank $2$.

\smallskip
\noindent 
{\bf Types $\ls3D_4$, $G_2$ and $\ls2G_2$.}
Fong and Milgram \cite{FM} studied in great detail the $2$-local
structure of $G$ in the case that $G$ has type $\ls3D_4$ or  $G_2$,
and described the structure of the
centralizers of the Klein four groups in a fixed Sylow $2$-subgroup of
$G$. They proved that these split into two conjugacy classes and that
their centralizers both have $2$-rank $3$.
While they assumed that $q\equiv 1\pmod4$, the Sylow $2$-subgroups are
isomorphic to those in the case where $q \equiv 3\pmod4$. So the same conclusion
is reached. A detailed description in the general case is in the paper
by Fong and Wong \cite{FW}. Note that $G_2(q)$ embeds in $\ls3D_4(q)$
as a subgroup of odd index, and hence their Sylow $2$-subgroups are
isomorphic (see also \cite[Theorem]{FW}).
We are left with the case of the groups $\ls2G_2(3^{2n+1})$. By \cite[Theorem 4.10.2(e)]{GLS} 
(see also \cite[Theorem 8.5]{ree61}),
a Sylow $2$-subgroup of $\ls2G_2(3^{2n+1})$ is elementary abelian of order
$8$, and so there are no maximal elementary abelian $2$-subgroups of
rank $2$.

This completes the proof of
Theorems \ref{thm:l=2} and \ref{thm:char2}.
\end{proof}


\section{When ${\mathbb G}$ is simple, $\ell = p$} 
\label{sec:defining}

When $\ell = p$, the structure of a Sylow $\ell$--subgroup 
of $G$ does not depend on the isogeny type. However, $\TF(G)$ 
can and does depend on the isogeny type because of  the fusion of $\ell$-subgroups. 
The following theorem summarizes the calculation of $\TF(G)$ in 
the defining characteristic. 


\begin{thm}\label{thm:l=p}
  Let $G$ be a finite group of Lie type, as in
  Definition~\ref{def:lietype}. Assume that the ambient algebraic
  group $\G$ is simple, and $\ell = p$.
 Then $\TF(G)\cong {\mathbb Z}$, provided $G$ is not one of the
 following types. 
\begin{itemize}
  \item[{}]
    \begin{AutoMultiColItemize}

\item[(a)] $\,\,A_{1}(p)$,
\item[(b)] $\!\!\ls2A_{2}(p)$,
\item[(c)] $\!\!\ls2B_2(2^{2a+1})$ (for $a\geq1$), 
\item[(d)] $\!\!\ls2G_2(3^{2a+1})$ (for $a\geq 0$), 
\item[(e)] $A_{2}(p)$,
\item[(f)] $B_{2}(p)$ and 
\item[(g)] $G_2(p)$.
 \end{AutoMultiColItemize}
\end{itemize}
In these exceptions, $\TF(G)$ is given in Tables 
\ref{sec:defining}.1 and \ref{sec:defining}.2.  
\end{thm}

We proceed to justify this result. 
For the simple algebraic group ${\mathbb G}$ fix an  
$F$-stable maximal split torus ${\mathbb T}$. 
Let $\Phi$ be the root system associated to 
$({\mathbb G},{\mathbb T})$. The positive (resp. negative)
roots are $\Phi^{+}$ (resp. $\Phi^{-}$), and $\Delta$ 
is a base consisting of simple roots.  

Let ${\mathbb B}$ be an $F$-stable Borel subgroup containing 
${\mathbb T}$ corresponding to the positive roots, and 
${\mathbb U}$ be the unipotent radical of ${\mathbb B}$. Then 
${\mathbb B}={\mathbb U}\rtimes {\mathbb T}$ with ${\mathbb B}$ and 
${\mathbb U}$ being $F$-stable. Set $B={\mathbb B}^{F}$ 
and $U={\mathbb U}^{F}$. 

There are three kinds of finite groups of Lie type $G$  
according to the type of $F$: (i) the untwisted groups, (ii) 
the twisted (Steinberg) groups and (iii) the very twisted groups 
(cf.\ \cite[Section 4]{CMN}, \cite[Section 2.3]{GLS}). 
In case (ii), $F$ involves 
a nontrivial graph automorphism 
$\tau$ of order $d$ of the underlying Dynkin 
diagram, as well as the Frobenius map. The 
automorphism $\tau$ induces a map from $\Phi$ to the 
{\em twisted root system\/} $\widetilde{\Phi}$ of $G$. 
Furthermore, we can define an equivalence relation on 
$\widetilde{\Phi}$ by identifying positive colinear 
roots, and let $\widehat{\Phi}$ be the set of equivalence 
classes. Therefore, we have mappings 
$\Phi\rightarrow \widetilde{\Phi}\rightarrow 
\widehat{\Phi}$. Let $\widehat{\Delta}$ be the image of $\Delta$ under this 
composition of maps and $\widetilde{\Delta}$ be the image of $\Delta$ under 
$\Phi\rightarrow \widetilde{\Phi}$. There are root 
subgroups of $G$ and these are indexed by the elements 
of $\widehat{\Phi}$. In the case that $G$ is 
untwisted then $\Phi=\widetilde{\Phi}=\widehat{\Phi}$. 
In case $G$ is a Steinberg group but not $~^2A_{2m}(q)$ we have
$\widetilde{\Phi}=\widehat{\Phi}$ 
(cf.\ \cite[Section 2.3]{GLS} for more details). 

As stated in the proof of \cite[Proposition~24.21]{MT},
there is a short exact sequence of groups 
\[
\xymatrix{1\ar[r]& Z^{F}\ar[r]& G_{sc}\ar[r]& G \ar[r]&Z_F\ar[r]&1}.
\]

In the case that $\ell=p$, $U$ is a Sylow $p$-subgroup of $G$. 
From \cite[Table 24.2]{MT}, $p$ does not divide $|Z^{F}|$. 
Therefore, the Sylow $p$-subgroups of $G_{sc}$ and of $G$ are 
isomorphic for any isogeny type, and so $\TF(U_{sc})\cong \TF(U)$. 

Given a finite group of Lie type $G$ where the underlying algebraic group 
is simple when ${\ell}=p$, one can make reductions to analyzing 
$\TF(G)$ in specific cases as follows. First, $\TF(G)\cong {\mathbb Z}$ 
when $|\widehat{\Delta}|\geq 3$ by \cite[Theorems 7.3 and 7.5]{CMN}. Note that 
the proofs of these results depend only on the 
structure of the Sylow $\ell$-subgroups. 
In the case when $|\widehat{\Delta}|=2$, by \cite[Theorems 7.3 and 7.5]{CMN}, 
$\TF(G)\cong {\mathbb Z}$ unless $G$ is $A_{2}(p)$, $B_{2}(p)$ or
$G_{2}(p)$. (Recall that we use the non-standard notation that e.g.,
$B_{2}(p)$  without any subscript denotes {\em any} group in this
isogeny class.)
The computation for $\TF(G)$ for these groups is given in Table  \ref{sec:defining}.1. 
\smallskip
\begin{center}
\begin{tabular}{  |l | c  | r | }
\hline 
\multicolumn{3}{|c |}
{Table  \ref{sec:defining}.1: \  $|\widehat{\Delta}|=2$} \\ \hline \hline
             $G$ & & rank $\TF(G)$ \\ \hline \hline
$A_{2}(p)_{sc}$      & $p=2$    &  2   \\ \hline
$A_{2}(p)_{sc}$      &  $p\geq 3$, $p \not\equiv 1\  (\text{mod} \ 3)$   &  3   \\ \hline
$A_{2}(p)_{sc}$      &   $p\geq 3$, $p \equiv 1\  (\text{mod} \ 3)$  &  5   \\ \hline
$A_{2}(p)_{ad}$      & $p=2$    &  2   \\ \hline
$A_{2}(p)_{ad}$     &   $p\geq 3$  & 3 \\ \hline
$B_{2}(p)$ &   $p=2, 3$    & 1  \\ \hline 
$B_{2}(p)$ &   $p\geq 5$  & 2 \\ \hline
$G_2(p)$ & $p=2,3,5$ &   1     \\  \hline
$G_2(p)$ & $p\geq 7$  & 2   \\ \hline
\end{tabular}
\end{center} 
\smallskip

Finally, in the case that $|\widehat{\Delta}|=1$, 
the Sylow ${\ell}$-subgroups are trivial intersection subgroups.  
The groups $G$ with $|\widehat{\Delta}|=1$ are 
$A_{1}(q)$, $^{2}A_{2}(q)$, $~^2B_2(2^{2a+1})$, 
and $~^2G_2(3^{2a+1})$. If $G=A_{1}(q)$ or $\ls2A_{2}(q)$
with $q>p$, the Sylow $p$-subgroups of $G$ 
have a noncyclic center, and therefore
$\TF(G)\cong {\mathbb Z}$ by Theorem~\ref{thm:poset2}. 
For the rest of the cases when 
$|\widehat{\Delta}|=1$, $\TF(G)$ is given 
in Table  \ref{sec:defining}.2 (cf.\ \cite[Section 5]{CMN}). 

\smallskip
\begin{center}
\begin{tabular}{ | l | c  | r | c  |l | c  | r | }
\hline
\multicolumn{3}{| c |}{Table  \ref{sec:defining}.2:\ \ $|\widehat{\Delta}|=1$} \\ \hline \hline
$G$ & & rank $\TF(G)$  \\ \hline \hline
$A_{1}(p)$ & $p\geq 2$  & 0  \\ \hline
$^{2}A_{2}(p)_{sc}$  &  $p=2$ & 0  \\ \hline 
$^{2}A_{2}(p)_{sc}$  & $p\geq 3$, $p \not\equiv -1\  (\text{mod} \ 3)$   & 1  \\ \hline
$^{2}A_{2}(p)_{sc}$  & $p\geq 3$,  $p \equiv -1\  (\text{mod} \ 3)$   & 3   \\ \hline 
$^{2}A_{2}(p)_{ad}$  &  $p=2$ & 0  \\ \hline 
$^{2}A_{2}(p)_{ad}$  &  $p\geq 3$ & 1  \\ \hline 
$~^2B_2(2)$   &       & 0 \\ \hline
$~^2B_2(2^{2a+1})$ & $a>0$      & 1 \\  \hline
$~^2G_2(3^{2a+1})$  & $a\geq 0$   & 1 \\ \hline
\end{tabular}
\end{center}
\smallskip

There is still some explanation needed to justify the data in the 
tables. We rely on some of the computations in \cite{CMN} in cases where there is one 
isogeny type. The results in \cite{CMN}  were only stated for the finite groups of Lie type arising from groups of adjoint isogeny type. 
Our new result, Theorem~\ref{thm:l=p}, extends to all finite groups of Lie type. 
We now proceed to dissect the cases when there is more than one isogeny type. 

For $A_{1}(p)$ a Sylow $p$-subgroup is cyclic of order $p$, and so
$\TF(G)$ does not depend on the isogeny type.
For $B_{2}(p)=C_{2}(p)$, we 
can use the calculations in \cite[Section 8]{CMN} 
which handle $B_{2}(p)_{sc}$ 
and $B_{2}(p)_{ad}$.   

Next we consider the case of $A_{2}(p)$ where there are two 
isogeny types. Let $U\cong U_{sc}\cong U_{ad}$ 
denote a Sylow $p$-subgroup in either type. 
The Sylow $p$-subgroup $U$ of $G$ is an 
extraspecial $p$-group of order $p^3$ and 
exponent $p$, if $p>2$. Moreover, if $p=2$ then $\SL_3(2)\cong \PSL_2(7)$ so $U$ is a dihedral 
group of order $8$, and has two maximal elementary abelian $2$-subgroups 
which are not conjugate in $U$ or in $G$. Consequently, $\TF(G)\cong {\mathbb Z}\oplus {\mathbb Z}$. 

If $p>2$, then all the elements of $U$ have order $p$, and there 
are $p+1$ maximal elementary abelian $p$-subgroups which are normal
in $U$. They can be described as follows. Let $\alpha$ 
and $\beta$ be simple roots so that
$U$ is generated by $x_{\alpha}$ and $x_{\beta}$. We have 
$$
x_{\alpha} = \begin{bmatrix} 1 &0 &0 \\  1& 1& 0 \\  0 & 0 &1 
\end{bmatrix}, \
x_{\beta} = \begin{bmatrix} 1 &0 &0 \\ 0& 1& 0 \\  0 & 1 &1 
\end{bmatrix}, \
x_{\alpha+\beta}= \begin{bmatrix} 1 &0 &0 \\ 0& 1& 0 \\  1 &0 &1
\end{bmatrix}.
$$
Then the 
elementary abelian $p$-subgroups in $U$ are the subgroups
$$
E_0 = \langle x_{\alpha}, x_{\alpha+\beta} \rangle, \ \ \ 
E_p = \langle x_{\beta}, x_{\alpha+\beta} \rangle, \ \ \  
E_i = \langle x_{\alpha}x_{\beta}^i, 
x_{\alpha+\beta}\rangle \ \ \text{for $i = 1, \dots, p-1$}.$$ 
Note that 
$$ 
x_{\alpha}x_{\beta}^{i}= \begin{bmatrix} 1 &0 &0 \\ 1& 1& 0\\  0 & i&1
\end{bmatrix}.
$$

Consider the action by conjugation of the group 
$D=\{t_{a, b, c}\mid\ a, b, c \in {\mathbb F}_{p}^\times\}$
where $t_{a,b,c}$ is the $3\times 3$ diagonal matrix with entries 
$a, b, c$. Let $I$ be the $3 \times 3$ identity matrix, and 
let $t= t_{a,b,c}$. We have that
\begin{equation} 
tx_{\alpha+\beta}t^{-1}= 
\begin{bmatrix} 
1 & 0 & 0\\  0 & 1& 0\\  a^{-1}c & 0  & 1
\end{bmatrix}, \ 
tx_{\alpha}t^{-1}= 
\begin{bmatrix} 
1 & 0 & 0\\  a^{-1}b & 1& 0\\  0 & 0  & 1
\end{bmatrix}, \ 
tx_{\beta}t^{-1}= 
\begin{bmatrix} 
1 & 0 & 0\\  0 & 1& 0\\  0 & b^{-1}c  & 1
\end{bmatrix},
\end{equation} 
\begin{equation} \label{eq:conjugation}
tx_{\alpha}x_{\beta}^it^{-1} = \begin{bmatrix} 
  1 & 0 & 0\\  a^{-1}b & 1 & 0\\  0 & b^{-1}ci &1 
\end{bmatrix}.
\end{equation} 
Under this action of $D$, it is easy to check that the 
subgroups $E_0, E_p$, and $E_1$ are in distinct $B$-conjugacy
classes. On the other hand, the set $\{E_i \mid i = 1, \dots, p-1 \}$
is a single $D$-conjugacy class because, given a nonzero $t\in\bF_p$ and setting
$a=b$, we can choose $0\neq c\in {\mathbb F}_{p}$ with $(b^{-1}c)i=t$. 
This shows that $\TF(B)\cong{\mathbb Z}^{\oplus 3}$ for 
$G=\text{PGL}_{3}(p)$. 

Now, set $\hat{D}=\{t_{a,b,c}\mid\ abc = 1\}$.
Then, in (\ref{eq:conjugation}), we set $a = b$, that is,
$b^{2}c=1$.
From (\ref{eq:conjugation}), we then have
$(b^{-1}c)i=b^{-3}i$. 
If $p-1\not\equiv 0 \pmod 3$, then for any
$t\in {\mathbb F}_{p}^\times$ there exists
$b\in {\mathbb F}_{p}^\times$ such that $(b^{-3})i=t$. 
If $p-1\equiv 0\pmod3$, let 
$\zeta$ be a generator of the multiplicative group ${\mathbb F}_{p}^{\times}$.
Then $b^{-3} \in \langle \zeta^{3} \rangle$ 
and every element in $\langle \zeta^{3} \rangle$ 
is a cube. So to determine the number of conjugacy classes 
of $\hat{D}$ on $\{E_i \mid i = 1, \dots, p-1 \}$ we use the 
equation $b^{-3}i=t$ to compute the number of cosets 
of $\langle \zeta^{3} \rangle$ in ${\mathbb F}_p^{\times}$ which turns out to be $3$. 
Therefore, in this case the number of $\hat{D}$-conjugacy classes of 
elementary abelian $p$-subgroups of rank $2$ in $U$ is $5$. 
In summary, for $G=\text{SL}_{3}(p)$ then 
$\TF(B)\cong{\mathbb Z}^{\oplus 3}$ (resp. ${\mathbb Z}^{\oplus 5}$) 
when $p\not\equiv1\pmod3$ (resp. $p\equiv1\pmod3$). 
Now we can show that $\TF(G)\cong \TF(B)$ by using the Bruhat decomposition. 

Next we consider the case of $\ls2A_{2}(p)$. When $p=2$, 
$U$ is a quaternion group and the $2$-rank of $U$ is $1$. Therefore, 
in this case $\TF(G)=\{0\}$. 

Now assume that $p\geq 3$. The case where 
$G=\text{SU}_{3}(p)$ was done in \cite[Section 5]{CMN}. 
This corresponds to $\ls2A_{2}(p)_{sc}$ 
(not $\ls2A_{2}(p)_{ad}$  which is incorrectly stated in \cite[Section 5]{CMN}). 
Now consider $G=\text{PGU}_{3}(p)$ for $p\geq 3$. As 
in the untwisted case we consider 
$D=\{t_{a,b,c}\mid\ a, b, c \in 
{\mathbb F}_{p^{2}}^{\times}\}$, and 
$D\cap \GU_{3}(p)$. The relations we obtain 
by intersecting are $ac^{p}=1$, $b^{p+1}=1$, and 
$ca^{p}=1$. In $U$ there are $p+1$ elementary abelian 
$p$-subgroups of $p$-rank $2$ given by 
$E_{i}=\langle x_{i},z \rangle$, $1\leq i \leq p+1$. Let $t$ be a 
generator for ${\mathbb F}_{p^{2}}^{\times}$.
The elements $x_{i}$ and $z$ are defined by  
\begin{equation} \label{eq:xi}
x_i=\begin{pmatrix} 1 & 0 & 0\\
t^i & 1 & 0\\
b_i & t^{ip} & 1\end{pmatrix}\quad\hbox{with $b_i+b_i^p=t^{i(p+1)}$}~,
\end{equation}

$$z=\begin{pmatrix} 1 & 0 & 0\\
0 & 1 & 0\\
u & 0 & 1\end{pmatrix}\quad\hbox{where $u\in{\mathbb F}_{p^2}$ 
satisfies $~u+u^p=0$}.
$$

 For any $j$, 
we can find $a\in {\mathbb F}_{p^{2}}^\times$ and $b, c$ such that 
$a^{-1}b=t^{j}$ satisfying the aforementioned 
relations as follows. Set $a=t^{(p-1)-j}$, $b=t^{p-1}$ and 
$c=t^{-((p-1)-j)p}$. Then  
\begin{equation} \label{eq:conjugation2}
t_{a,b,c}x_{i}t_{a,b,c}^{-1} =\begin{bmatrix} 1 & 0 & 0\\  
a^{-1}bt^{i}& 1& 0\\  
a^{-1}cb_{i} & b^{-1}ct^{ip} &1 
\end{bmatrix}
=\begin{bmatrix} 1 & 0 & 0\\  t^{i+j}& 1& 0\\  
a^{-1}cb_{i} & t^{(i+j)p} &1 
\end{bmatrix}.  
\end{equation} 
One can verify that $a^{-1}cb_{i}$ satisfies the equation in (\ref{eq:xi}) with $i$ replaced with $i+j$. 
This shows that under conjugation by elements in $D\cap \text{GU}_{3}(p)$, 
there is a single conjugacy class among 
$\{E_{i}\mid 1\leq i \leq p+1\}$. Hence, for 
$G=\text{PGU}_{3}(p)$ with $p\geq 3$, $\TF(G)\cong {\mathbb Z}$.

\section{Extending the results from simple to reductive groups}
\label{sec:reductive}

Let $G = \G^F$ be a finite group of Lie type arising
from a connected reductive algebraic
group $\G$ and  a Steinberg endomorphism $F$ of $\G$.
In this section, we show that the torsion free rank of the group of
endotrivial modules of $G$ can be obtained by considering the
components of the decomposition of $\mathbb G$ as a product of simple
algebraic groups. Our detailed analysis completes the proofs 
of Theorems \ref{thm:tf-reductive-main}
and \ref{thm:main-grp}. 

From \cite[1.8]{Car}, we have that $\G=[\G,\G]\cdot\S$ where
the derived subgroup $[\G,\G]$ is semisimple
and $\S=Z(\G)^{0}$ is the connected center of $\G$.
The intersection of these groups $Z = [\G,\G] \cap \S$ is a finite group.
Therefore, we have an exact sequence
\begin{equation} \label{eq:reductalg}
\xymatrix{
1 \ar[r] &  Z \ar[r] & [\G,\G] \times \S \ar[r] & \G \ar[r] &  1.
}
\end{equation}

Set $G = {\G}^F$ and  $G_{ss}  = {[\G,\G]}^F$.
Upon taking fixed points, one obtains an
exact sequence (cf.\ \cite[Lemma 24.20]{MT})
\begin{equation} \label{eq:ses-reductive}
\xymatrix{
1 \ar[r] & Z^F \ar[r] &  G_{ss} \times \S^F
\ar[r]^{\quad \psi} & G \ar[r] & Z_F \ar[r] & 1
}
\end{equation}
with $Z_F$ denoting co-invariants. Here,
$\psi$ is injective on restriction to both
$G_{ss}$ and $\S^F$.

Since $[\G,\G]$ is semisimple one can express
$[\G,\G] = \H_1  \cdots \H_s$ where each
$\H_i$ is a central product of $n_i$
isomorphic simple algebraic groups
$\K_i$ where $F$ preserves $\H_i$ and
$\H_i^F \cong \K_i^{F^{n_i}}$ \cite[Proposition 2.2.11]{GLS},
the fixed points of $\K_i$ under $F^{n_i}$.
So there is an exact sequence
\begin{equation}\label{eq:ses-prod}
\xymatrix{
1 \ar[r] &  A \ar[r] & \H_1 \times \cdots \times \H_s \ar[r] &
[\G,\G]  \ar[r] & 1
}
\end{equation}
for a finite abelian group $A$ of order prime to $p$. Once
again, we apply \cite[Lemma 24.20]{MT} to get the
exact sequence
\begin{equation}\label{ss-seq}
\xymatrix{
1 \ar[r] &  A^F  \ar[r] & \H_1^F \times \cdots \times \H_s^F
\ar[r] &  G_{ss} \ar[r] & A_F \ar[r] & 1.
}
\end{equation}
For each $i$,  set $H_i = \H_i^F \leq G_{ss}$.
In addition, we have the following statements.
\begin{itemize}
\item[(i)] $\vert Z_F \vert = \vert Z^F \vert$ and
$\vert A_F \vert = \vert A^F \vert$.
\item[(ii)] Suppose that $x$ is an element in $G$ that it not in $G_{ss}$. For
any $i$, conjugation by $x$ preserves $H_i$.
Moreover, if $H_i$ is isomorphic to $\SL_n(q)$, $\SU_n(q)$ or $\Sp_n(q)$, then
$x$ induces on $H_i$ an
automorphism that coincides with conjugation by an element in (respectively)
$\GL_n(q)$, $\GU_n(q)$ or  $\CSp_n(q)$.
\end{itemize}

The equalities in (i) are consequences of the fact that the order of a finite
group of Lie type is independent of the isogeny type, a consequence of
the order formula \cite[Corollary~24.6]{MT}.
For (ii), let $x\in G$ with $x\notin G_{ss}$.
From (\ref{eq:reductalg}), $x=gz$ where
$g\in [\G,\G]$ and $z\in \S$ with $z\neq 1$.
Here $F(x) = x$, so that $g^{-1}F(g) = zF(z^{-1})$.
Moreover, from (\ref{eq:ses-prod}), $g=h_{1}h_{2}\dots h_{s}$ with
$h_{j}\in \mathbb H_j$
for $j=1,2,\dots,s$. Because $z$ is central and
$H_1\cdots H_s$ is a central product,
action of conjugation by $x$ on $H_i$ is the same as
conjugation by $h_i$. Thus $h_i$ is an element of $\mathbb{H}_i$
that normalizes $H_i$.  As explained in \cite[Proposition~2.5.9(b)]{GLS},
this means that $h_i$ lies in the preimage of $(\H_i/Z)^F$ in $\H_i$,
with $Z$ a central subgroup of
$\H_i$. Now, if $H_i$ is $\SL_n(q)$, $\SU_n(q)$ or
$\Sp_n(q)$, then we can without restriction assume that $\K_i$ is
either $\SL_n$ or $\Sp_n$. Let $\tilde \K_i$ be $\GL_n$ and $\CSp_n$
respectively, and let $\tilde \H_i$ be the corresponding central
product, constructed as for $\H_i$. Note that $\H_i \leq \tilde \H_i$,
that the central subgroup $\tilde Z$ of $\tilde \H_i$ is connected, and that $(\H_i/Z)^F \cong (\tilde \H_i
/\tilde Z)^F$. The preimage of $(\tilde \H_i
/\tilde Z)^F$ in $\tilde \H_i$ equals $\tilde \H_i^F \tilde Z$, as
$\tilde Z$ is connected, so $h_i \in \tilde \H_i^F \tilde Z$.
Hence, $h_i$, and therefore $x$, induce the same conjugation on $H_i$  as an element in $\tilde
\H_i^F$, which is what we claimed in (ii).
The main theorem of this section is the following.

\begin{thm} \label{thm:tf-reductive}
Suppose that $G$ is a finite group of Lie type
with $G = \G^F$ for $\G$ a connected
reductive algebraic group over an algebraically closed field
of characteristic $p$, and $F$ a Steinberg endomorphism.
Assume that $\TF(G)$ has rank greater than $1$.

If $\ell \neq p$
then $G \cong U \times V$
where $V$ has order prime to $\ell$ and $\TF(G) \cong \TF(U)$. Moreover,
\begin{itemize}
\item[(a)] if $2 < \ell \neq p$ then $U$ is one of the groups listed
in Theorem \ref{thm:lgeq3}, and
\item[(b)] if $\ell = 2 \neq p$ then $U$ is one of the groups listed
in Theorem \ref{thm:l=2} and $V$ is abelian.
\end{itemize}
In the event that $\ell =p$, then
$G/Z(G) \cong H/Z(H)$, where $H$ is one of
the groups in
 Tables
 \ref{sec:defining}.1 and \ref{sec:defining}.2.
\end{thm}

The proof is divided into three cases. First we deal with $\ell = p$, and
then with $\ell \neq p$, which is again divided into two steps
depending on whether $\ell$ is
odd or even.
Throughout the proof we employ the conventions introduced prior to the
theorem.


Observe first that if $G = U \times V$, and $\ell$ does not divide
$\vert V \vert$, then the restriction map provides an isomorphism  $\TF(G) \xrightcong \TF(U)$.
This is because, in this case, any endotrivial
$kU$-module becomes an endotrivial $kG$-module on
inflation, so the restriction map $T(G) \to
T(U)$ is surjective; and it has finite kernel, again
because the index of $U$ in
$G$ is prime to $\ell$.

\begin{proof}[Proof of Theorem~\ref{thm:tf-reductive} when $\ell = p$]
In this case the groups $Z^F$ and $Z_F$ have order relatively
prime to $\ell$. Hence, $\psi$ induces an isomorphism on Sylow
$\ell$-subgroups. Note that, as we are in the defining characteristic,
$\ell$ divides the order of each $H_i$. However, then $s=1$ in \eqref{ss-seq}, as otherwise a Sylow $\ell$-subgroup
$S$ of $G$ would split as a non-trivial direct product implying $\TF(G)
\cong \Z$ by Lemma~\ref{lem:normal3rk}. This also means that $A =1$,
and $G_{ss} = H_1$.
We have a central extension
$ 1 \to \S \to  \G \to \G/\S \to 1$
producing on fixed-points another central extension
$$ 1 \to \S ^F \to  \G^F \to (\G/\S)^F \to 1$$
where $(\G/\S)^F \cong \K^{F^{n_1}}$ for some simple algebraic group
$\K$ by \cite[Proposition
2.2.11]{GLS}. Now set $H = \K^{F^{n_1}}$ so that $G/Z(G) \cong
H/Z(H)$. Observe that $TF(G) \xrightcong
TF(H)$ by Proposition~\ref{P:groupexactseq}. Hence, 
Theorem~\ref{thm:l=p} says that $H$ is one of the groups listed in 
Tables \ref{sec:defining}.1 and \ref{sec:defining}.2.
\end{proof}


\begin{proof}[Proof of Theorem \ref{thm:tf-reductive}
when $3 \leq \ell \neq p$] Assume that $\TF(G)$ is not cyclic.

\smallskip 
\noindent{\sc Step 1:} We prove first that the prime $\ell$ does not 
divide $\vert H_i \vert$ for more than one $i$. 
Assume that $TF(G)$ is not cyclic and that there is more than 
one $H_i$ whose order is divisible by $\ell.$ Note that $\ell$ 
has to divide $\vert Z(H_i)
\vert$  every time it divides $|H_i|$, since otherwise 
a Sylow $\ell$-subgroup
$S$ of $G$ splits as a non-trivial direct factor 
implying that $Z(S)$ has
$\ell$-rank at least $2$. This means that
we are done by Lemma \ref{lem:normal3rk}. 
The tables of centers of the finite groups of Lie type
(cf.\ \cite[Table 24.2]{MT}) show that if
$\ell$ divides $\vert Z(H_i) \vert$, then $H_i$ has one of the types:
$A_{n-1}(q)$ for $\ell \mid (n,q-1)$, $\ls2A_{n-1}(q)$
for $\ell \mid (n,q+1),$
$E_6(q)$ with $\ell = 3$, or  $\ls2E_6(q)$ with $\ell = 3$. Hence, we
can assume that $H_i$ is one of these types when $\ell$ divides $\vert
Z(H_i) \vert$. The two last cases, involving the groups of type $E$, can
furthermore be eliminated, using Theorem \ref{T:norank2}, as
the $3$-ranks of $E_6(q)$ and $\ls2E_6(q)$ are $6$.

We now deal with the groups of type $A$.
Because $\ell$ divides $n$, the $\ell$-ranks of these groups
is at least $\ell-1$. Therefore, if we have more than one $H_i$ of order
divisible by $\ell$, and none of the groups split off as a direct
factor, the $\ell$-rank of the resulting group will be at least
$(\ell-1)+(\ell-1) - 1 = 2\ell-3$. This number has to be at most
$\ell$ by Theorem \ref{T:norank2}. So we conclude that the only 
possibility is that $\ell = 3$ and $n = 2$, assuming that $\ell$
divides the order of the center of $H_i$.

Note that if there is an $H_i$ whose order is 
not divisible by $3$, then $H_i$ is a Suzuki group (Lie type $\ls2B_2$),
and these groups have trivial centers. So for the purposes of our argument, 
we may assume that there are exactly two components $H_1$ and $H_2$
both having order divisible by $3$. Moreover, because $Z(H_1)$ and
$Z(H_2)$ are not trivial we have that these groups must be the finite
groups arising from the simply connected algebraic groups: 
$H_i = \SL_3(q_i)$ where $3$ divides $q_i-1$,
or $H_i = \SU_3(q_i)$ with $3$ dividing $q_i+1$. Let $3^{t_i}$ be the 
highest power of $3$ dividing $q_i-1$ in the first case and dividing
$q_i+1$ in the second. 

In the exact sequence (\ref{ss-seq}), the image of the group $A^F$ is 
central in $H_1 \times H_2$ and hence it must have order either $1$ or $3$. 
Similarly in sequence (\ref{eq:ses-reductive}), the image of $Z^F$ in 
$H_1H_2 = G_{ss}$ is central and its order is either $1$ or $3$. We claim 
first that if $A^F = \{1\}$, then we are done. The reason is that then 
$G_{ss} \cong H_1 \times H_2$ which has $3$-rank $4$. The map $\psi$
is injective on $G_{ss}$, so that $G$ also has $3$-rank $4$, and we are 
finished by Theorem \ref{T:norank2}(a). Hence, $G_{ss}= H_1H_2$ is the 
central product of $H_1$ and $H_2$ over a central subgroup of order $3$.

Let $S_i$ be a Sylow $3$-subgroup of $H_i$ and $S$ a Sylow $3$-subgroup
of $G$.  Each $S_i$ can be chosen to have a maximal toral subgroup 
$T_i = C_{3^{t_i}} \times C_{3^{t_i}}$ of diagonal matrices with an 
element of order 3 in the form of a permutation matrix acting on it. 
Thus its center has order $3^{t_i}$. 

Suppose that $\vert Z^F \vert = 1$. In the event that both $t_1$ and
$t_2$ are greater than $1$, there are elements $y_1 \in Z(S_1)$
and $y_2 \in Z(S_2)$ having order $9$ such that $y_1^3= z_1$ 
and $y_2^3= z_2$
are the central elements in $H_1$ and $H_2$ that are identified when 
$A^F$ is factored out. Thus the classes of $y_1y_2^{-1}$ and $z_2$ 
modulo $A^F$ are in the center of $S$ and the center of $S$ has $3$-rank 
two. Consequently, we are done in this case and we may assume that $t_1 = 1$.  

Still assuming that $\vert Z^F \vert = 1$, we are down to the situation
that $S_1$ is an extraspecial group of order $27$ and exponent $3$. 
If the class of $(x,y) \in S_1 \times S_2$ modulo $A^F$ has order $3$,
then $(x,y)^3 = (1,y^3) \in A^F$ and $y$ has order $3$. Thus 
the class of $(x,y)$ modulo $A^F$ commutes with 
those of $(x,1)$ and $(1,y)$. In this way we see that the centralizer of
every element of order $3$ in $S$ has $3$-rank at least $3$, and we 
are done with this case. 

We conclude that $\vert Z^F \vert = 3$ and we can assume that $S$
is an extension:
\[
\xymatrix{
1 \ar[r] & S_1S_2 \ar[r] & S \ar[r] & Z_F \ar[r] & 1
}
\]
where $Z_F$ is cyclic of order $3$. From the above arguments, we know that
the centralizers of elements of order $3$ in $S_1S_2$ have $3$-rank $3$.
For the purposes of this proof,  assume that $H_i \cong \SL_3(q_i)$.
Let $x \in S$ be an element of order $3$ that is not in $S_1S_2$. Then 
$x$ must act on $S_1$ as conjugation by 
an element of $\GL_3(q_1)$. So its action element
is conjugate (by an element $\SL_3(q_1)$) to an element of the diagonal
torus. Therefore, its centralizer $K_1$ 
in $H_1 \cong \SL_3(q_1)$ has $3$-rank $2$. The same 
happens for the centralizer $K_2$ of its action on $H_2$. By a similar argument,
the same condition holds when $H_1$ or $H_2$ is isomorphic to $\SU_3(q)$.
It follows that the subgroup of $G$ generated
by $x$, $K_1$ and $K_2$ has $3$-rank at least $4$. Hence, $G$ has 
$3$-rank at least $4$ and we are done by Theorem \ref{T:norank2}(a).
This completes the first step.

\smallskip\noindent
{\sc Step 2:} In this step we complete the proof assuming that 
$\ell$ divides $\vert H_1 \vert$ and does not divide $\vert H_i \vert$
for $i > 1$. Assume that $TF(G)$ has rank greater
than $1$. We wish to show that $G$ has the form $U \times
V$, where $V$ has order prime to $\ell$ and $U$ is one of the groups
listed in Theorem~\ref{thm:lgeq3}.

If $\ell \nmid |Z(H_1)|$, then a Sylow $\ell$-subgroup of $H_1$ is a
direct factor in some Sylow $\ell$-subgroup of $G$. As the $\ell$-part
of the center of a Sylow $\ell$-subgroup of $G$ is 
cyclic if the rank of $TF(G)$ is greater than one, we conclude 
that $|\S^F|$ is prime to $\ell$. Hence, $G$ has
the same $\ell$-local structure as $H_1$. 
Theorem~\ref{thm:lgeq3} now shows that $H_1$ is
isomorphic to one of the groups listed in that theorem. In particular
$Z(H_1)=1$, so $G \cong H_1 \times V$ for some $\ell'$-group $V$, as
asserted. 

Next suppose that $\ell \mid |Z(H_1)|$.
Our aim is to prove that there are no groups with
$TF(G)$ having rank greater than one that can occur, thus finishing 
the proof in the case that $\ell \geq 3$. 
First note that, with our assumptions,  
$G$ has the same $\ell$-local structure as
$(\G/(\H_2\cdots\H_s))^F$, and that the $\ell$-part of $\S^F$
is cyclic as the $\ell$-part of $Z(G)$ is.
The rank argument
from Step 1 shows that $H_1$ must have Lie type $A$. More precisely, we
must have $H_1 \cong \SL_{\ell}(q)$ with $\ell \mid (q-1)$ or $H_1 \cong
\SU_\ell(q)$, with $\ell \mid (q+1)$. The sequence
\eqref{eq:ses-reductive} shows that the $\ell$-local structure
of $G$ must agree with that of a central product $ \langle H_1,
\zeta\rangle \Delta$ where $\zeta$ is an element with determinant of
order $\ell$ inside $\GL_\ell(q)$ or $\GU_\ell(q)$, $\Delta$ is
cyclic of order $\ell^t$, for some $t$, and $ \langle H_1,
\zeta\rangle \cap \Delta$ has order $\ell$. However,
such a group has the same poset of conjugacy classes of
elementary abelian $\ell$-subgroup as $\langle H_1, \zeta\rangle$,
which is an associated group as defined in Section~\ref{sec:assoc}.
Hence,  the torsion free rank of the group of endotrivial modules
cannot be larger than one, as the group does not appear in
Theorem~\ref{thm:assoc}.
\end{proof}

\begin{proof}[Proof of Theorem~\ref{thm:tf-reductive}
when $2 = \ell \neq p$]

Assume first that $s>1$ and that $\TF(G)$ has rank greater
than $1$. We want to show that this case cannot occur.
Observe first that every factor $H_i$, being a nonabelian finite group of Lie
type, has even order, as does $H_i/Z(H_i)$.
In addition, the order of the center of any factor must be even,
as otherwise a Sylow $2$-subgroup of $H_i$ is a direct factor of some
Sylow 2-subgroup of $G$ and hence its center has $2$-rank greater than $1$.
As a result we can assume that
every $H_i$ has type $A_n$, for $n$ odd, $B_n$, $C_n$, $D_n$ or
$E_7$ by the table of orders of centers in \cite[Table 24.2]{MT}.

Recall that by Theorem \ref{T:norank2},
the sectional $2$-rank of $G$ can not be $5$ or
more. The group $G$ contains the direct product
$H_1/Z(H_1) \times \dots \times H_s/Z(H_s)$ as a section. From the
proof of Theorem \ref{thm:char2}, we know that the sectional $2$-rank of a group of type 
$A_1$ or $\ls2A_1$ is two, while the sectional $2$-rank of a group of 
type $A_n$ or $\ls2A_n$ for $n \geq 3$ is at least $3$. In addition, the 
sectional $2$-ranks for groups of types $B_n$, $C_n$, $D_n$ and 
$E_7$ are at least $3$. As a result, the only possible situation with 
sectional $2$-rank less than $5$ occurs when there are exactly two 
components $H_1$ and $H_2$ both of type $A_1$ or $\ls2A_1$. We 
henceforth assume that this is the situation. 

Because $\psi$ is injective on restriction to $\S^F$. It must be
that $Z^F$ is either trivial or has order $2$. In addition,
the image $W$ of the inclusion of
$Z^F$ into $G_{ss} \times \S^F$ followed by the projection onto 
$\S^F$ must be the Sylow $2$-subgroup of $\S^F$. The reason is 
that otherwise, the quotient group $G_{ss}/Z(G_{ss}) \times \S^F/W$,
which is a section of $G$, has sectional $2$-rank $5$ and by 
Theorem \ref{T:norank2}(b), $\TF(G) \cong \Z$. 
If $Z^F$ is trivial, then so is $Z_F$ and a Sylow $2$-subgroup 
$S$ of $G$ is either a direct product or a central product of 
quaternion groups. In the first case, $Z(S)$ has 
$2$-rank $2$ and we are done by Lemma \ref{lem:normal3rk}. A direct
calculation shows that the all maximal elementary abelian $2$-subgroups
of a central product of quaternion groups have $2$-rank $3$.

Hence, we may assume that $Z^F$ has order $2$ and that $S$ is an
extension (cf. the exact sequence~\eqref{eq:ses-reductive})
\[
\xymatrix{
1 \ar[r] & S_1S_2 \ar[r] & S \ar[r] & C_2 \ar[r] & 1
}
\]
where $S_1$, $S_2$ are normal quaternion subgroups and $S_1S_2$ is a 
central product. We have noted already that the centralizer of any 
involution in $S_1S_2$ has $2$-rank $3$. We need only show the same 
for any involution $x$ not in $S_1S_2$. 
The involution $x$ must act on each $S_i$ as an element of $\GL_2(q)$,
which means that it must normalize, but not centralize, some (necessarily
cyclic, since $S_{i}$ are quaternion) subgroup $\langle y_1 \rangle$ of order $4$ in $S_1$ and another 
$\langle y_2 \rangle$ in $S_2$. But then $y_1^2 = y_2^2$ is the
nontrivial central
element in $S_1S_2$, and hence $y_1y_2$ is a noncentral involution in the 
centralizer of $x$. So we have shown $c_{G}(x)$ has $2$-rank at least three. Therefore, we have reduced ourselves to situation where 
$s =1$. 

Now assume that $s = 1$. We follow the pattern of Step $2$ of the proof 
in the case that $p \neq \ell \geq 3$. As shown in that proof, we may assume
that $\ell=2$ divides the order of $Z(H_1)$, 
as otherwise $G \cong H_1 \times V$
where $H_1$ is one of the listed groups. In addition we may assume that 
$H_1$ has sectional $2$-rank at most $4$. The combination of the conditions
that $2 \mid \vert Z(H_1) \vert$ and that the sectional rank be less than $5$,
means that $H_1$ must have one of the types $A_1$, $\ls2A_1$, $A_3$, $\ls2A_3$
or $B_2$ (see Theorem \ref{thm:largerank} and \cite[Table 24.2]{MT}).  
Then as in Step $2$ of the odd characteristic case, the $2$-local structure 
of $H_1$ is that of a central product. Note that in the case that $H_1$ has
type $B_2$ and $H_1 = \Sp_4(q)$, then the element $\zeta$ has order $2$ 
in $\CSp_4(q)$. We note also that if $H_i$ has type $A_3$, and $q \equiv 1$
modulo $4$, then a Sylow $2$-subgroup of $H_1$ has a rank $3$ torus that is 
a characteristic subgroup. It follows that $TF(G) \cong \Z$, as we
have seen before. The same  
happens if $H_1$ has type $\ls2A_3$ and $q \equiv 3 \pmod 4$. Hence, the 
only possibilities are that $H_1$ is one of $\SL_2(q) \cong \SU_2(q)$, 
$\SL_4(q)$ with $q \equiv 3 \pmod 4$, $\SU_4(q)$ with $q \equiv 1 \pmod 4$
or $\Sp_4(q)$. As before we conclude that the group $G$ has the same poset
of conjugacy classes of elementary abelian $2$-subgroups as an associated group
to $H_1$ as defined in Section \ref{sec:assoc}. In the case that $\ell =2$
these groups were treated in Section \ref{sec:le2generic}. In particular, 
Theorem \ref{thm:char2} is sufficient to finish the proof. 
\end{proof}

This finishes the proof of Theorem~\ref{thm:tf-reductive}. We now
verify that this indeed proves the main theorems.

\begin{proof}[Proof of Theorems \ref{thm:tf-reductive-main}
  and \ref{thm:main-grp}]
  First recall that Theorem~\ref{thm:main-grp} is equivalent to
Theorem~\ref{thm:tf-reductive-main} by Theorem~\ref{thm:poset2}, where
in Theorem~\ref{thm:main-grp} we have sorted the list by $\ell$-rank
instead of by prime.
To verify  Theorem \ref{thm:tf-reductive-main}, suppose that $TF(G)$ has rank greater
than one.

If $\ell \neq p$ and $\ell > 2$, then
Theorem~\ref{thm:tf-reductive}(a) says that $G \cong H \times K$
where $\ell \nmid |K|$ and $H$ is listed in
Theorem~\ref{thm:lgeq3}, which is the list in Theorem
\ref{thm:tf-reductive-main}\eqref{case-cross} with $\ell \neq 2$.

If
$\ell \neq p$ and $\ell =2$ then Theorem~\ref{thm:tf-reductive}(b)
tells us that $G \cong H \times K$ with $\ell \nmid |K|$ and $H \cong \PGL_2(q) \cong
\PGU_2(q)$, which is the list in Theorem
\ref{thm:tf-reductive-main}\eqref{case-cross} with $\ell = 2$.

Now suppose that $\ell = p$.  Then the last part of
Theorem~\ref{thm:tf-reductive} says that $G/Z(G) \cong H/Z(H)$, where $H$ is
one of the groups in Theorem~\ref{thm:l=p} with the rank of $TF(H)$
greater than $1$. An inspection of Tables 1 and 2 now shows that $H$
is either ${}^2A_2(p)_{sc}$ with $3 \mid p+1$, $A_2(p)_{sc}$,
$A_2(p)_{ad}$, $B_2(p)_{sc}$ with $p \geq 5$,  $B_2(p)_{ad}$ with $p
\geq 5$, or $G_2(p)$ with $p \geq
7$. This produces the list for $G/Z(G) \cong H/Z(H)$ given in
Theorem~\ref{thm:tf-reductive-main}\eqref{case-char}, by translating
into classical group notation.

The theorems and tables quoted in Theorem~\ref{thm:tf-reductive-main}
give the indicated ranks,
finishing the proof of that theorem. 
\end{proof}

\end{document}